\documentclass[11pt,reqno]{amsart}
\usepackage{xifthen,setspace}
\usepackage{pkgfile}
\newtheorem{example}{Example}

\allowdisplaybreaks

\title[Second Moment Phenomenon for Monochromatic Subgraphs]{The Second Moment Phenomenon for Monochromatic Subgraphs}

\author{Bhaswar B. Bhattacharya}
\address{Department of Statistics, University of Pennsylvania, Philadelphia, USA,
{\tt bhaswar@wharton.upenn.edu}}

\author{Somabha Mukherjee}
\address{Department of Statistics, University of Pennsylvania, Philadelphia, USA,
{\tt somabha@wharton.upenn.edu}}

\author[Sumit Mukherjee]{Sumit Mukherjee\textsuperscript{*}}\thanks{\textsuperscript{*}Research partially supported by NSF grant DMS-1712037}
\address{Department of Statistics, Columbia University, New York, USA, {\tt  sm3949@columbia.edu}}

\begin{document}

%\begin{abstract}
%
%\end{abstract}


\begin{abstract} What is the chance that among a group of $n$ friends, there are $s$ friends all of whom have the same birthday? This is the celebrated birthday problem which can be formulated as the existence of a monochromatic $s$-clique $K_s$ ($s$-matching birthdays)  in the complete graph $K_n$, where every vertex of $K_n$ is uniformly colored with $365$ colors (corresponding to birthdays). More generally, for a general connected graph $H$, let $T(H, G_n)$ be the number of monochromatic copies of $H$ in a uniformly random coloring of the vertices of the graph $G_n$ with $c_n$ colors.   In this paper we show that $T(H, G_n)$ converges to $\dPois(\lambda)$ whenever $\E T(H, G_n) \rightarrow \lambda$ and $\Var T(H, G_n) \rightarrow \lambda$, that is, the asymptotic Poisson distribution of $T(H, G_n)$ is determined just by the convergence of its mean and variance. Moreover, this condition is necessary if and only if $H$ is a star-graph. 
In fact, the second-moment phenomenon is a consequence of a more general theorem about the convergence of $T(H,G_n)$ to a finite linear combination of independent Poisson random variables. As an application, we derive the limiting distribution of $T(H, G_n)$, when $G_n\sim G(n, p)$ is the Erd\H os-R\'enyi random graph. Multiple phase-transitions emerge as $p$ varies from 0 to 1, depending on whether the graph $H$ is balanced or unbalanced.   
\end{abstract}

%\spacing{1.125}

\maketitle

\section{Introduction}

Let $G_n$ be a simple labeled undirected graph with vertex set $V(G_n):=\{1,2,\cdots,|V(G_n)|\}$, edge set $E(G_n)$, and adjacency matrix $A(G_n)=\{a_{ij}(G_n),i,j\in  V(G_n)\}$.  In a {\it uniformly random $c_n$-coloring of $G_n$}, the vertices of $G_n$ are colored with $c_n$ colors as follows:  
\begin{equation}\P(v\in V(G_n) \text{ is colored with color } a\in \{1, 2, \ldots, c_n\})=\frac{1}{c_n},
\label{eq:uniform}
\end{equation}
independent from the other vertices. Let $X_v$ denote the color of the vertex $v \in V(G_n)$ in a uniformly random $c_n$-coloring of $G_n$. A subgraph $F$ of $G_n$ with vertex set $V(F) = \{v_1,\dots,v_{|V(F)|}\}$ is said to be {\it monochromatic} if $X_{v_1}=\dots=X_{v_{|V(F)|}}$. 

In this paper we consider the problem of determining the limiting distribution of the number of monochromatic copies of a general connected simple graph $H$, in a uniformly random $c_n$-coloring of a graph sequence $G_n$. Formally, this is defined as
\begin{align}\label{eq:TH}
T(H, G_n):=\frac{1}{|Aut(H)|}\sum_{\bm s \in V(G_n)_{|V(H)|}} \prod_{(a,b) \in E(H)}a_{s_a s_b}(G_n) \bm 1\{X_{=\bm s}\},
\end{align}
where:  
\begin{itemize}
\item[--] For a finite set $S$ and a positive integer $N$, $S_N$ denotes the set of all $N$-tuples ${\bm s}=(s_1,\cdots, s_{N})\in  S^N$ with distinct entries.\footnote{For a set $S$, the set $S^N$ denotes the $N$-fold cartesian product $S\times S \times \cdots \times S$.} Thus, the cardinality of $S_N$ is $\frac{|S|!}{(|S|-N)!}$. 
\item[--] For any ${\bm s}=(s_1,\cdots, s_{|V(H)|}) \in  V(G_n)_{|V(H)|}$, 
\begin{align*}%\label{eq:sind}
\bm 1\{X_{=\bm s}\}:= \bm 1\{X_{s_1}=\cdots=X_{s_{|V(H)|}}\}.
\end{align*}
\item[--] $Aut(H)$ is the {\it automorphism group} of $H$, that is, the group of permutations $\sigma$ of the vertex set $V(H)$ such that $(x, y) \in E(H)$ if and only if $(\sigma(x), \sigma(y)) \in E(H)$. 
\end{itemize}
 
Hereafter, we assume that $H$ is simple and connected,  with $|V(H)| \geq 2$, and $V(H) = \{1,2,\cdots,$ $|V(H)|\}$. Note that for the case $H=K_2$ is an edge, the statistic \eqref{eq:TH}  counts the number of monochromatic edges. This statistic arises in several contexts, for example, as the Hamiltonian of the Ising/Potts models on $G_n$ \cite{bmpotts}, in non-parametric two-sample tests \cite{fr}, and as a generalization of the birthday paradox \cite{barbourholstjanson,dasguptasurvey,diaconisholmes,diaconismosteller}: If $G_n$ is a friendship-network graph colored uniformly with $c_n=365$ colors (corresponding to birthdays and  assuming the birthdays are uniformly distributed across the year), then two friends will have the same birthday whenever the corresponding edge in the graph $G_n$ is monochromatic.\footnote{When the underlying graph $G_n=K_n$ is the complete graph $K_n$ on $n$ vertices, this reduces to the classical birthday problem.}  Therefore, $\P(T(K_{2}, G_n)>0)$ is the probability that there are two friends with the same birthday. Note that $\P(T(K_{2}, G_n)>0)=1-\P(T(K_{2}, G_n)=0)=1-\chi_{G_n}(c_n)/c_n^{|V(G_n)|}$, where $\chi_{G_n}(c_n)$ counts the number of proper colorings of $G_n$ using $c_n$ colors. The function $\chi_{G_n}$ is known as the {\it chromatic polynomial} of $G_n$, and is a central object in graph theory \cite{chromaticbook,toft_book,toft_unsolved}.  More generally, $s$-matching birthdays in a friendship network $G_n$ corresponds to the case $H=K_s$ (the complete graph on $s$ vertices) in \eqref{eq:TH}. The asymptotics of birthday collisions have found many applications, for example, in the study of coincidences \cite[Problem 3]{diaconismosteller}, hash-function attacks in cryptology \cite{ns}, and the discrete logarithm problem \cite{ghdl,pollarddlp}. 
 
In this paper, we study the asymptotic distribution of $T(H, G_n)$, in the regime where $\E(T(H, G_n))=O(1)$. It is well-known that the limiting distribution of $T(K_{2}, G_n)$,  exhibits a {\it first-moment phenomenon}, that is, $T(K_2, G_n)\dto \dPois(\lambda)$, for any graph sequence $G_n$ such that $\E(T(K_2, G_n))=\frac{1}{c_n} |E(G_n)| \rightarrow \lambda$.  This was shown by Barbour et al. \cite[Theorem 5.G]{barbourholstjanson}, using the Stein's method for Poisson approximation. Recently, Bhattacharya et al. \cite[Theorem 1.1]{BDM} gave a new proof of this result based on the method of moments, which illustrates interesting connections to extremal combinatorics. The first-moment phenomenon is not true for general graphs $H$: it is easy to construct examples where $\E(T(H, G_n)) \rightarrow \lambda $, but $T(H, G_n) \nrightarrow \dPois(\lambda)$ \cite[Section 8]{BDM}, if $H \ne K_2$. In this paper, we show that the limiting distribution of $T(H, G_n)$, for a general connected graph $H$, exhibits a {\it second-moment phenomenon}:  $T(H, G_n) \dto \dPois(\lambda)$ whenever $\E T(H, G_n) \rightarrow \lambda$ and $\Var T(H, G_n) \rightarrow \lambda$, that is, the limiting Poisson distribution of $T(H, G_n)$ is determined by the convergence of its mean and variance. This complements and generalizes the result for $T(K_2, G_n)$, since, in this case, the variance condition $\Var T(K_2, G_n) \rightarrow \lambda$ is automatically implied by the mean condition $\E T(K_2, G_n) \rightarrow \lambda$. Using this result, the limiting distribution of $T(H, G_n)$ in the Erd\H os-R\'enyi random graph is derived, where interesting phase-transitions emerge.

\subsection{The Second Moment Phenomenon} 

Throughout the paper, we will assume that $H$ is a finite, simple, and connected graph with no isolated vertices, and $G_n$ a sequence of growing simple graphs, with the vertices colored uniformly with $c_n$ colors. We will also assume that $c_n \rightarrow \infty$ as $n \rightarrow \infty$, unless specified otherwise.

\begin{thm}\label{thm:poisson} Let $H\ne K_2$ be as above, and $\{G_n\}_{n \geq 1}$ a sequence of graphs colored uniformly with $c_n$ colors, such that 
\begin{align}\label{eq:expvar}
\lim_{n \rightarrow \infty}\E T(H, G_n)=\lambda \quad \text{and}\quad \lim_{n \rightarrow \infty}\Var T(H, G_n)=\lambda.
\end{align}
Then $T(H, G_n) \dto \dPois(\lambda)$. 
\end{thm}

Note that Theorem \ref{thm:poisson} assumes that $H\ne K_2$, which corresponds to monochromatic edges. In this case, it is easy to check that  
$$\E(T(K_2, G_n))=\frac{|E(G_n)|}{c_n} \quad \text{and} \quad \Var(T(K_2, G_n))=\frac{|E(G_n)|}{c_n}\Big(1-\frac{1}{c_n}\Big).$$  
Therefore, the assumption $\E(T(K_2, G_n)) \rightarrow \lambda$ automatically ensures that $\Var(T(K_2, G_n)) \rightarrow \lambda$. As a consequence, the variance condition \eqref{eq:expvar} cannot be leveraged, when $H=K_2$, and the proof presented in this paper breaks down. However, as mentioned earlier, the conclusion in Theorem \ref{thm:poisson} still holds when $H=K_2$, that is, $T(K_2, G_n) \dto \dPois(\lambda)$, whenever $\E(T(K_2, G_n)) \rightarrow \lambda$ (refer to \cite[Theorem 5.G]{barbourholstjanson} and \cite[Theorem 1.1]{BDM} for two different proofs of this result). Therefore, {\it the second-moment phenomenon} holds for all connected graph $H$, that is, the limiting Poisson distribution of the $T(H, G_n)$ is determined by the convergence of its first two moments.

The proof of Theorem \ref{thm:poisson} is described in Section \ref{sec:pfpoisson_linear}. In fact, this theorem is a consequence of a more general result (Theorem \ref{thm:poisson_linear}) where we derive a general sufficient condition under which $T(H, G_n)$ is a finite linear combination of independent Poisson random variables.  The proof is based on a truncated moment-comparison technique, and has two main steps:
\begin{itemize}
\item[--] We begin with a truncation step: This involves defining a {\it remainder term}, which (informally) counts the number of tuples $\bm s \in V(G_n)_{|V(H)|}$ such that the number of copies of $H$ passing through a subset of indices in $\bm s$ is `large'. The first step is to show that the remainder term converges to zero in $L_1$, because of the variance assumption in \eqref{eq:expvar} (Lemma \ref{lm:residual}). 

\item[--] To analyze the {\it main term}, which is $T(H, G_n)$ minus the remainder term,  we use the `independent approximation', which shows that the moments of the random variable obtained by replacing the indictors $\bm 1\{X_{=\bm s}\}$ by independent $\dBer(\frac{1}{c_n^{|V(H)|-1}})$ variables, for every subset of vertices in $G_n$ of size $|V(H)|$, are asymptotically  close (Lemma \ref{lm:momentdiff}). The result then follows by deriving the asymptotic distribution of the approximating variable, which is a finite linear combination of independent Bernoulli random variables, each of which converges to a Poisson distribution (Lemma \ref{lm:Wpoisson}). 
\end{itemize}

The truncation step is necessary because, $T(H, G_n)$, for a general graph $H$, does not converge in moments (see Theorem \ref{thm:star} below), and hence, its limiting distribution, cannot be  captured by a direct moment-based argument.

\begin{remark} Another natural approach to proving a limiting Poisson distribution is through the Stein's method for Poisson approximation \cite{poisson2,barbourholstjanson,birthdayexchangeability,CDM}.  In fact, the well-known Stein's method based on dependency graphs \cite[Theorem 15]{CDM}, bounds the convergence rate in terms of covariances  (but, not in terms of the mean and the variance). Arratia et al. \cite{arratia} used this to obtain rates of convergence for the number of monochromatic cliques in a uniform coloring of a complete graph (see also Chatterjee et al. \cite{CDM}). However, this cannot be used to prove Theorem \ref{thm:poisson} for a general graph $H$, as the condition imposed by the convergence of the mean and the variance is, in general, weaker than what is required by a generic dependency graph construction (refer to Remark \ref{rm:trianglestein} for a specific example). Moreover, our general result (Theorem \ref{thm:poisson_linear}) goes beyond the Poisson regime, and captures the asymptotic regime where $T(H,G_n)$ is a finite linear combination of Poisson variables. 
\end{remark}

Next, we consider the converse to Theorem \ref{thm:poisson}, that is, whether the Poisson convergence of $T(H, G_n)$ implies the convergence of the first two moments. The following theorem shows that this is true if and only if $H$ is a {\it star-graph}, that is, $H = K_{1, r}$ for some integer $r \geq 1$.  
 
\begin{thm}\label{thm:star} Fix an integer $r\geq 1$, a real number $\lambda > 0$, and a sequence of graphs $\{G_n\}_{n \geq 1}$ colored uniformly with $c_n$ colors. Then  $T(K_{1, r}, G_n) \dto \dPois(\lambda)$ if and only if 
\begin{align}\label{eq:rcondition}
\lim_{n \rightarrow \infty}\E T(K_{1, r}, G_n)=\lambda \quad \text{and}\quad \lim_{n \rightarrow \infty}\Var T(K_{1, r}, G_n)=\lambda.
\end{align}
Moreover, if $H$ is connected and is not a star-graph, then there exists a sequence of graphs $\{G_n(H)\}_{n \geq 1}$ such that $T(H, G_n(H)) \dto \dPois(\lambda)$, but \eqref{eq:rcondition} does not hold. 
\end{thm}

The proof of the theorem is given in Section \ref{sec:pfstar}. In fact, the proof shows that when $H$ is a star-graph, we have convergence in all moments, that is, \eqref{eq:rcondition} implies that $T(K_{1, r}, G_n) \rightarrow \dPois(\lambda)$ in distribution and in all moments, and conversely, $T(K_{1, r}, G_n)$ converges in distribution to $\dPois(\lambda)$ implies the convergence of all moments of $T(K_{1, r}, G_n)$ to the corresponding moments of $\dPois(\lambda)$. 

\begin{remark} The second moment phenomenon for the Poisson distribution complements the well-known {\it fourth-moment phenomenon}, which asserts that the limiting normal distribution of certain homogeneous forms is implied by the convergence of the corresponding sequence of fourth moments (refer to Nourdin et al. \cite{greater_three} and the references therein, for general fourth-moment theorems and invariance principles, and Bhattacharya et al. \cite[Theorem 1.3]{BDM} for an example of this phenomenon in random graph coloring). In this regard, it would be interesting to see if the Poisson second-moment phenomenon extends beyond monochromatic subgraphs to general integer-valued homogeneous forms.
\end{remark}

\subsection{Application to Erd\H os-R\'enyi Random Graphs}

Theorem \ref{thm:poisson} can be easily extended to random graphs, when the limits in Theorem \ref{thm:poisson_linear} hold in probability, under the assumption that the graph and its coloring are jointly independent (see Lemma \ref{lm:random} for details). Using this we can derive the limiting distribution of $T(H, G_n)$, where $G_n\sim G(n, p)$ is the Erd\H os-R\'enyi random graph, colored uniformly with $c_n$ colors (independently of the graph), such that \begin{equation}\label{expcopy}
\E T(H,G_n) = \frac{|V(H)|! {n \choose |V(H)|} p^{|E(H)|}}{|Aut(H)|c_n^{|V(H)|-1}} \rightarrow \lambda.
\end{equation}    
This implies $c_n=\Theta(n^{\frac{|V(H)|}{|V(H)|-1}} p^{\frac{|E(H)|}{|V(H)|-1}})$. Also the condition $c_n\rightarrow\infty$ implies $n^{\frac{|V(H)|}{|E(H)|}}p\rightarrow\infty$.

Under the above scaling, Theorem \ref{thm:poisson_linear} can be used to characterize the limiting distribution of $T(H, G_n)$ for all connected graphs $H$, where $G_n\sim G(n, p)$ and $p=p(n) \in (0, 1)$. Here multiple interesting phase transitions occur depending on whether the graph $H$ is balanced or unbalanced. We begin by recalling the notion of balancedness of a graph. 

\begin{defn}\cite[Chapter 3]{random_graphs_janson} For a finite connected graph $H$, define 
\begin{align}\label{eq:mH}
m(H)=\max_{H_1\subseteq H}\frac{|E(H_1)|}{|V(H_1)|},
\end{align}
where the maximum is over all non-empty subgraphs $H_1$ of $H$. The graph $H$ is said to be {\it balanced}, if $m(H)=\frac{|E(H)|}{|V(H)|}$, and {\it unbalanced} otherwise. Moreover, the graph $H$ is said to be {\it strictly balanced} if $\frac{|E(H')|}{|V(H')|} < \frac{|E(H)|}{|V(H)|}=m(H)$, for all proper subgraphs $H'$ of $H$. 
\end{defn}

In the balanced case, where the asymptotic distribution of $T(H,G_n)$ undergoes a phase transition from $\dPois(\lambda)$ to a linear combination of independent Poissons, depending on whether $p(n) \rightarrow 0$ or $p(n):=p$ is fixed, respectively. 

\begin{thm}\label{thm:balanced} (Balanced Graphs) Let $H$ be a simple connected balanced graph, and $G_n\sim G(n,p)$ be the Erd\H os-R\'enyi random graph, with $p:=p(n) \in (0, 1)$, colored uniformly with $c_n$ colors such that \eqref{expcopy} holds. Then the following cases arise: 
\begin{enumerate}

\item[(a)] If $n^{-\frac{|V(H)|}{|E(H)|}}\ll p(n)\ll 1$, then $T(H,G_n)
 \dto \dPois(\lambda)$.

\item[(b)] If $p(n) := p \in (0, 1)$ is fixed, 
\begin{align}\label{eq:pdense}
T(H, G_n) \dto \sum_{F \supseteq H: |V(F)|=|V(H)|} N(H, F) X_F,
\end{align} 
where $X_F\sim \dPois\left(\lambda\cdot \frac{|Aut(H)|}{|Aut(F)|}p^{|E(F)|-|E(H)|}(1-p)^{{|V(H)| \choose 2}-|E(F)|}\right)$ and the collection $\{X_F: F \supseteq H \text{ and } |V(F)|=|V(H)|\}$ is independent. 
\end{enumerate}
\end{thm}

Note that the sum  in \eqref{eq:pdense} above  is over the set of non-isomorphic (unlabelled) graphs $F$, which contain $H$ as a subgraph and has the same number of vertices as $H$. The proof of Theorem \ref{thm:balanced} is given in Section \ref{sec:erpf}. 

The situation, however, is more delicate for unbalanced graphs. To explain this, we need the following definition:  

\begin{defn}
For an unbalanced graph $H$, define the exponent 
\begin{align}\label{eq:expp}
\gamma(H)&:=\min_{H_1\subset H}\frac{|V(H)|-|V(H_1)|}{|E(H_1)|(|V(H)|-1)-|E(H)|(|V(H_1)|-1)}, 
\end{align}
where the minimum is over the set of all proper subgraphs $H_1$ of $H$, for which the denominator is positive. 
\end{defn}

%%%%%%%%%%%%%%%%%%%%%%%%%%%%%%%%%%%%%%%%%%%%%%%%%%%%%%%%%%%%%%%%%%%%%%%%%%%%%%%%
\begin{figure}[h]
\centering
\begin{minipage}[l]{1.0\textwidth}
\centering
\includegraphics[width=6.25in]
    {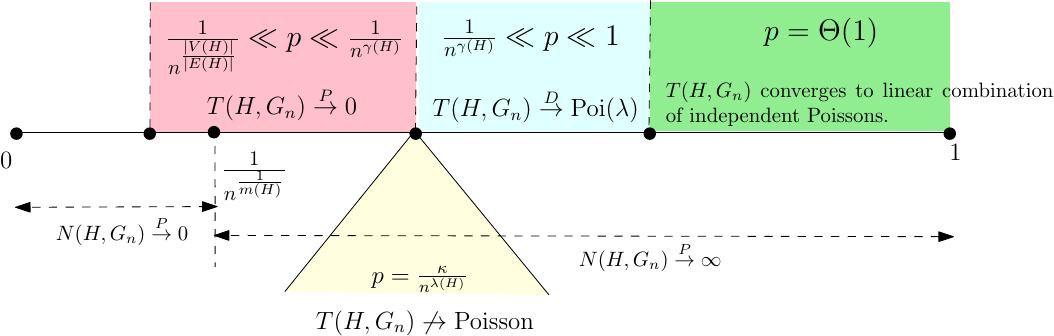}\\
%\vspace{-0.5in}
%\small{(a)}
\end{minipage} 
\caption{\small{Phase transitions of $T(H, G_n)$, for an unbalanced graph $H$ in the Erd\H os-R\' enyi random graph $G_n\sim G(n, p)$, as $p$ varies from 0 to 1.}}
\label{fig:erdos_renyi}
\end{figure}
%%%%%%%%%%%%%%%%%%%%%%%%%%%%%%%%%%%%%%%%%%%%%%%%%%%%%%%%%%%%%%%%%%%%%%%%%%%%%%%%%%%%%%%%%%%%%%%%%%%%%%%%%%

It is easy to verify that $\gamma(H)$ is well-defined and positive, for any unbalanced graph $H$ (see Lemma \ref{lm:egamma}).\footnote{Even though, for our results, we only need to define $\gamma(H)$ for unbalanced graphs, it is natural to wonder what happens to the quantity in the RHS of \eqref{eq:expp} for balanced graphs. We show in Lemma \ref{lm:egamma}, for $H$ balanced, but not strictly balanced, $\gamma(H)$ as in \eqref{eq:expp}, is well-defined and equals to $\frac{1}{m(H)}$. On the other hand, if $H$ is strictly balanced, there are cases where the RHS of \eqref{eq:expp} is finite, and there are cases where it is undefined. } When $H$ is unbalanced, the asymptotic distribution of $T(H, G_n)$, where $G_n\sim G(n, p(n))$, undergoes an additional phase-transition, whose location is determined by the exponent $\gamma(H)$.

\begin{thm}(Unbalanced Graphs)\label{thm:unbalanced}
 Let $H$ be a simple connected unbalanced graph, and $G_n\sim G(n,p)$ be the Erd\H os-R\'enyi random graph, with $p:=p(n) \in (0, 1)$, colored uniformly with $c_n$ colors,  such that \eqref{expcopy} holds. Then the following cases arise: 

\begin{enumerate}
\item[(a)] If $n^{-\frac{|V(H)|}{|E(H)|}} \ll p(n) \ll n^{-\gamma(H)}$,  then $T(H, G_n)\pto 0$. 

\item[(b)] If $n^{\gamma(H)}p(n)\rightarrow \kappa\in (0,\infty)$ then all moments of $T(H, G_n)$ converges. Moreover,  if $T(H, G_n)$ converges in distribution to a random variable $W$, then $W$ is not Poisson.

\item[(c)] If $n^{-\gamma(H)} \ll p(n)\ll 1$, then $T(H,G_n) \dto \dPois(\lambda)$.

\item[(d)] If $p(n) := p \in (0, 1)$ is fixed, then $T(H, G_n)$ converges to the RHS of \eqref{eq:pdense}, that is, a linear combination of independent Poisson random variables. 
\end{enumerate}
\end{thm}

The proof of Theorem \ref{thm:unbalanced} is given in Section \ref{sec:erpf}. The phase transitions of $T(H, G_n)$, for an unbalanced graph $H$, are shown in Figure \ref{fig:erdos_renyi}.

\begin{remark} It is well-known that $n^{-\frac{1}{m(H)}}$ is the threshold for the occurrence of $H$ in the random graph $G(n, p)$ \cite[Theorem 3.4]{random_graphs_janson}. Therefore, for unbalanced graphs, since $\gamma(H) < \frac{1}{m(H)}$ (Lemma \ref{lm:egamma}), there exists an interesting regime ($n^{-\frac{1}{m(H)}} \ll p \ll n^{-\gamma(H)}$) where $N(H, G_n)$, the number of copies of $H$ in $G_n$, goes to infinity, but the number of monochromatic copies $T(H, G_n)$ converges in probability to zero, that is, we do not  have convergence of moments. Another surprising feature of unbalanced graphs is that the asymptotic distribution of $T(H, G_n)$ transitions from being degenerate at zero (equivalently, $\dPois(0)$) to $\dPois(\lambda)$, through a non-Poisson limit at the point of criticality ($p=\frac{\kappa}{n^{\gamma(H)}}$).  It remains open to show that the limit of $T(H, G_n)$ exists at the critical point, and finding the limiting distribution? Preliminary calculations in a few examples seem to suggest that the limiting moments may not satisfy Stieltjes moment condition \cite{Stieltjes}, and so we cannot conclude  existence of limiting distribution from the convergence of moments. 
\end{remark}

\subsection{Organization} The rest of the paper is organized as follows: The general limiting distribution of monochromatic subgraphs and the proof of Theorem \ref{thm:poisson} are given in Section \ref{sec:pfpoisson_linear}. The proof of Theorem \ref{thm:star} is given in Section \ref{sec:pfstar}. Applications to the Erd\H os-R\'enyi random graph (proofs of Theorem \ref{thm:balanced} and Theorem \ref{thm:unbalanced}) and the birthday problem are discussed in Section \ref{sec:examples}.

%%%%%%%%%%%%%%%%%%%%%%%%%%%%%%%%%%%%%%%%%%%

\section{Limiting Distribution of Monochromatic Subgraphs}
\label{sec:pfpoisson_linear}

In this section we derive general sufficient conditions under which the random variable $T(H, G_n)$ converges to a linear combination of independent Poisson random variables.  We begin with a few definitions and notations: For a finite simple unlabeled graph $F$, denote by $\hom_{\mathrm{inj}}(F, G_n)$ the set of injective homomorphisms from $F$ to $G_n$, that is, the set of injective maps $\phi: V (F) \rightarrow V (G_n)$, such that $(\phi(x), \phi(y)) \in E(G_n)$ whenever $(x, y) \in E(F)$. It is easy to see that $$
|\hom_{\mathrm{inj}}(H, G_n)|=\sum_{\bm s \in V(G_n)_{|V(H)|}} \prod_{(a,b) \in E(H)}a_{s_a s_b}(G_n).$$  Moreover, denote by $N(F, G_n)$ the number of copies of $F$ in $G_n$, and $N_{\mathrm{ind}}(F, G_n)$ the number of induced copies of $F$ in $G_n$. Note that 
\begin{align}\label{eq:poissonexp}
N(H, G_n)=\frac{|\hom_{\mathrm{inj}}(H, G_n)|}{|Aut(H)|}\quad \text{and} \quad \E(T(H, G_n)) =\frac{N(H, G_n)}{c_n^{|V(H)|-1}}.
\end{align}

Next, we introduce the notion of join of two graphs. These graphs will show up in the analysis of the variance of $T(H, G_n)$. 

\begin{defn}\label{defn:tjoin} Fix $t \in [1, |V(H)|]$. Let $H'$ be an isomorphic copy of $H$, with  $V(H)=\{1, 2, \ldots, |V(H)|\}$ and $V(H')=\{1', 2', \ldots, |V(H)|'\}$, where $z' \in V(H')$ is the image of $z \in V (H)$.  For two ordered index sets $J_1=(j_{11}, j_{12}, \ldots, j_{1t}) \in [|V(H)|]_t$ and $J_2=(j_{21}, j_{22}, \ldots, j_{2t}) \in [|V(H)|]_t$, denote by $H_t(J_1, J_2)$ the simple graph obtained by the union of $H$ and $H'$, when the vertex $j_{1a} \in V(H)$ is identified with the vertex $j_{2a}' \in V(H')$, for $a \in [t]$. More precisely,
$$H_t(J_1, J_2)=\left(V(H) \bigcup \gamma(V(H')), E(H) \bigcup \gamma(E(H'))\right),$$ 
where
\begin{itemize}

\item[--] $\gamma(V(H'))=\{\gamma(v'): v' \in V(H')\}$, where  $\gamma$ is a relabelling of the vertices of $V(H')$ such that $\gamma(j_{2a}')=j_{1a}$, for $a \in [t]$, and $\gamma(v')=v'$, for $v' \notin J_2$. 

\item[--] This induces a relabelling of the edges $\gamma(E(H'))=\{\gamma((u', v')): (u', v') \in E(H')\}$, where $\gamma((u', v'))=(\gamma(u'), \gamma(v'))$, for $(u', v') \in E(H')$.

\end{itemize}
{\it The graph $H_t(J_1, J_2)$ will be referred to as the $t$-join of $H$ with pivots at $J_1$ and $J_2$} (see Figure \ref{fig:Hdefn}). Denote by $\sJ_t(H):=\{H_t(J_1, J_2): J_1, J_2 \in [|V(H)|]_t\}$ the collection of all graphs (up to isomorphism) which can be obtained as the $t$-join of $H$. Finally, a graph $F$ is said to be a {\it join} of two isomorphic copies of $H$, if $F \in \sJ_t(H)$, for some $t \in [1, |V(H)|]$. 
\end{defn}

%%%%%%%%%%%%%%%%%%%%%%%%%%%%%%%%%%%%%%%%%%%%%%%%%%%%%%%%%%%%%%%%%%%%%%%%%%%%%%%%
\begin{figure*}[h]
\centering
\begin{minipage}[l]{1.0\textwidth}
\centering
\includegraphics[width=5.55in]
    {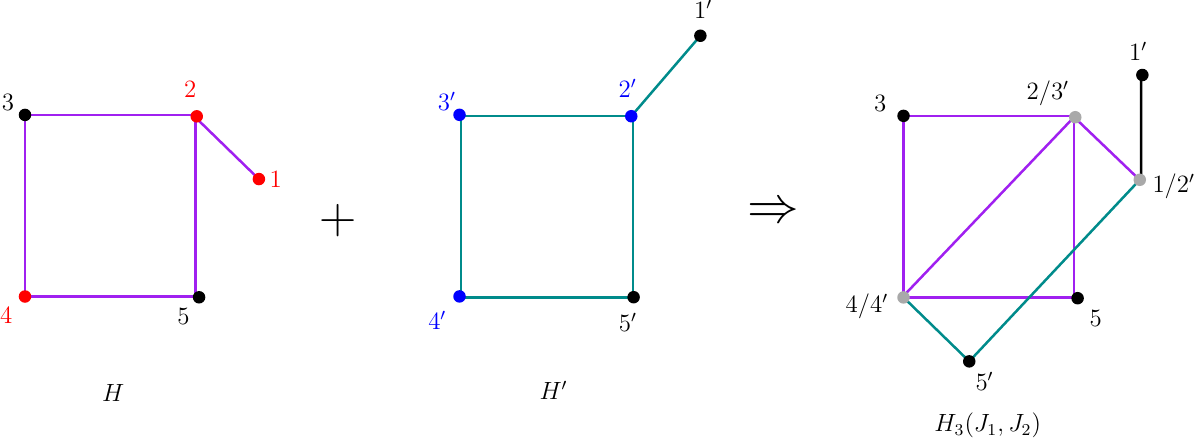}\\
%\vspace{-0.5in}
%\small{(a)}
\end{minipage} 
\caption{\small{3-join of $H$ (the 4-cycle with an edge hanging from one vertex of the cycle) with pivots $J_1=(1, 2, 4)$ and $J_2=(2, 3,4)$.}}
%\vspace{-0.15in}
\label{fig:Hdefn}
\end{figure*}
%%%%%%%%%%%%%%%%%%%%%%%%%%%%%%%%%%%%%%%%%%%%%%%%%%%%%%%%%%%%%%%%%%%%%%%%%%%%%%%%%%%%%%%%%%%%%%%%%%%%%%%%%%

Equipped with the above definitions, we can now state our general theorem:

\begin{thm}\label{thm:poisson_linear} Let $H$ be as in Theorem \ref{thm:poisson}, and $G_n$ be a sequence of  graphs colored uniformly with $c_n$ colors, such that the following hold:

\begin{itemize}

\item[--] For every $k \in [1,N(H,K_{|V(H)})]$, there exists $\lambda_k \geq 0$ such that
\begin{equation}\label{eq:Finduced}
	\lim_{n \rightarrow \infty}  \frac{\sum_{F\in \sC_{H,k}} N_{\mathrm{ind}}(F,G_n)}{c_n^{|V(H)|-1}} = \lambda_k ,
\end{equation}
where $\sC_{H,k} := \{F \supseteq H: |V(F)|=|V(H)| \text{ and } N(H,F)=k\}$.\footnote{Note that the graphs in the set $\sC_{H,k}$ are unlabelled. In other words, $\sC_{H,k}$ is the collection of non-isomorphic graphs with the same number of vertices as $H$ and containing $k$ copies of $H$.}

\item[--] For $t \in [2, |V(H)|-1]$ and every $F \in \sJ_t(H)$, as  $n \rightarrow \infty$, 
$N(F, G_n)=o(c_n^{2|V(H)|-t-1}).$ 
\end{itemize}
Then 
\begin{align}\label{eq:Tlinear}
T(H, G_n) \dto \sum_{k=1}^{N(H,K_{|V(H)|})} k Z_k,
\end{align}
where $Z_k\sim \dPois(\lambda_k)$ and the collection $\{Z_k: 1\leq k\leq N(H,K_{|V(H)|})\}$ is independent. 
\end{thm}

The second condition ensures that the counts all sub-graphs of $G_n$ which arise as the join of two non-disjoint copies of $H$ on non-identical vertex sets,  (that is, $t \ne \{1, |V(H)|\}$) are asymptotically negligible. Moreover, as $\Cov(\bm 1\{X_{=\bm s}\}, \bm 1\{X_{=\bm t}\})=0$, whenever $\bm s, \bm t \in V(G_n)_{|V(H)|}$  have at most 1 index in common, the only terms in $\Var T(H, G_n)$ which contribute are those which arise as a $|V(H)|$-join of two copies of $H$. Therefore, Theorem \ref{thm:poisson_linear} captures the asymptotic regime where $T(H, G_n)$ is `linear', and to ensure the existence of the limiting distribution we assume \eqref{eq:Finduced}.

\begin{remark} An easy sufficient condition for \eqref{eq:Finduced} is the convergence of $\frac{1}{c_n^{|V(F)|-1}} N_{\mathrm{ind}}(F, G_n)$ for {\it every} super-graph $F$ of $H$ with $|V(F)|=|V(H)|$. However, condition \eqref{eq:Finduced} does not require the convergence for every such graph, and is applicable to more general examples, as described below: Define a sequence of graphs $G_n$ as follows:
\[   
G_n = 
\begin{cases}
\text{ disjoint union of } n \text{ isomorphic copies of } C_4&\quad\text{ if } ~n \text{ is odd}\\
\text{ disjoint union of } n \text{ isomorphic copies of } \cD&\quad\text{ if } ~n \text{ is even.}\\ 
\end{cases}
\]
where $C_4$ denotes the $4$-cycle and $\cD$ is the $4$-cycle with one diagonal. Choosing $c_n = \lfloor n^{1/3} \rfloor$, gives $\E(T(C_4, G_n))\rightarrow 1$. In this case, 
$$ \frac{\sum_{F\in \sC_{H,1}} N_{\mathrm{ind}}(F,G_n)}{c_n^3} = \frac{N_{\mathrm{ind}}(C_4,G_n)+N_{\mathrm{ind}}(\cD,G_n)}{c_n^3} \rightarrow 1,$$ 
and $\frac{1}{c_n^3} \sum_{F\in \sC_{H,3}} N_{\mathrm{ind}}(F,G_n)=\frac{1}{c_n^3}N_{\mathrm{ind}}(K_4, G_n)=0$, and $\frac{1}{c_n^3} \sum_{F\in \sC_{H,2}} N_{\mathrm{ind}}(F,G_n)=0$, since $\sC_{H,2}$ is empty. Therefore, Theorem \ref{thm:poisson_linear} implies that $T(C_4, G_n) \dto \dPois(1)$ (which can also be directly verified, because, in this case, $T(C_4, G_n)$ is a sum of independent $\dBer(\frac{1}{c_n^3})$ variables).  However, it is easy to see that individually both $\frac{1}{c_n^3} N_{\mathrm{ind}}(C_4, G_n)$ and $\frac{1}{c_n^3} N_{\mathrm{ind}}(\cD, G_n)$ are non-convergent.  \end{remark}

\begin{remark}
Note that a linear combination of Poisson random variables is a special case of the discrete compound Poisson distribution \cite{poisson_zl}. To this end, denote by $Z$ the random variable in the RHS of \eqref{eq:Tlinear} and let $\kappa_H=N(H, K_{|V(H)|})$. Then define $Y$ to be a discrete random variable with $$\P(Y =  k) =  \frac{\lambda_k}{\sum_{k=1}^{\kappa_H} \lambda_k}, \quad \text{for } k \in \{1, \ldots,  \kappa_H \}.$$ It is then easy to see that $Z$ has the same distribution as the discrete compound Poisson variable $Z' = \sum_{i=1}^N Y_i$,  where $\{Y_1, Y_2, \ldots\}$ are independent copies of $Y$, and $N \sim \dPois(\sum_{k=1}^{\kappa_H} \lambda_k)$, which is independent of $\{Y_1, Y_2, \ldots\}$.  
\end{remark}

The rest of this section is organized as follows: The proof of Theorem \ref{thm:poisson_linear} is given below in Section \ref{sec:pfpoissonlinear} and the proof of Theorem \ref{thm:poisson} is described in Section \ref{sec:pfpoisson}.

\subsection{Proof of Theorem \ref{thm:poisson_linear}}
\label{sec:pfpoissonlinear}

We begin with a few notations and definitions. For an ordered tuple $\bm t$ with distinct entries, denote by $\bar{\bm t}$ the (unordered) set formed by the entries of $\bm t$ (for example, if $\bm t=(4, 2, 5)$, then $\bar{\bm t}=\{2, 4, 5\}$).  

Given $J\subseteq V(H)$, define $H[J]$ to be the induced subgraph of $H$ on the vertices in $J$, $H\backslash J$ the graph obtained by removing all vertices in $J$ and the associated edges, and $E_H(J, J^c) = \{(x,y)\in E(H): x\in J \text{ and } y\in V(H)\setminus J\}$. Clearly, $E(H)=E(H[J]) \bigcup E(H\backslash J) \bigcup E_H(J, J^c)$ is an edge partition of $E(H)$.

\begin{defn} Fix $t \in [2, |V(H)|]$ and $J =(j_1, j_2, \ldots, j_t) \in V(H)_t$ and $\bm r=(r_1, r_2, \ldots, r_t) \in V(G_n)_{t}$. Denote by $M_J(\bm r, H, G_n)$  the number of injective homomorphism $\phi: V(H) \rightarrow V(G_n)$ such that $\phi(j_a)=r_a$, for $a \in [t]$. More formally, define $\psi: \bar J \rightarrow [t]$ as $\psi(j_b)=b$, for $b \in [t]$, then 
$$M_J(\bm r, H, G_n)=\prod_{(x, y) \in E(H[\bar{J}])} a_{r_{\psi(x)} r_{\psi(y)}}(G_n) \sum_{\substack{\bm s_{J^c}}} \prod_{(x, y) \in E_{H}(\bar{J}, \bar{J}^c) } a_{r_{\psi(x)} s_y}(G_n) \prod_{(x, y) \in E(H\backslash \bar{J})} a_{s_x s_y}(G_n),$$
where the sum is over tuples $s_{J^c} := (s_x)_{x\in V(H)\setminus \bar{J}} \in (V(G_n)\setminus \bar{\bm r})_{|V(H)\setminus \bar{J}|}$. 
\end{defn}

\begin{example}\label{ex:truncation} To help parse the above definition, we compute $M_J(\cdot, H, G_n)$ in a few examples: 
\begin{itemize}

\item[--]$H=K_{1, 2}$ is the 2-star with the central vertex labeled 1 and $J=(2, 3)$. Then with $\bm r=(i, j)$, 
$$M_{(2, 3)}((i, j), K_{1, 2}, G_n)= M_{(3, 2)}((i, j), K_{1, 2}, G_n) = \sum_{\substack{s=1\\ s \notin \{i, j\}}}^{|V(G_n)|} a_{i s}(G_n) a_{j s}(G_n):=t_{G_n}(i, j),$$ where $t_{G_n}(i, j)$  is the number of common neighbors of $i, j$. Similarly, 
\begin{align}\label{eq:MK12}
M_{(1, 2)}((i, j), K_{1, 2}, G_n)=M_{(1, 3)}((i, j), K_{1, 2}, G_n) = a_{ij}(G_n) (d_{G_n}(i)-a_{ij}(G_n)),
\end{align}  
where $d_{G_n}(i)$ denotes the degree of the vertex $i$ in $G_n$. Finally, 
$M_{(2,1)}((i, j), K_{1, 2}, G_n)=M_{(3,1)}((i, j), K_{1, 2}, G_n) = a_{ij}(G_n) (d_{G_n}(j)-a_{ij}(G_n))$.

\item[--]$H=P_4$, the path of length 3, with vertices labeled $\{1, 2, 3, 4\}$ in order and $J=(2, 4)$, then with $\bm r=(i, j)$, 
$$M_{(2, 4)}((i,j), P_4, G_n)=\sum_{\substack{s_1=1\\ s_1 \notin \{i, j\}}}^{|V(G_n)|} \sum_{\substack{s_3=1\\ s_3 \notin \{s_1, i, j\}}}^{|V(G_n)|} a_{s_1 i}(G_n) a_{i s_3}(G_n) a_{s_3 j}(G_n).$$
The expressions for other ordered tuples $J$ can be obtained similarly. 
\end{itemize}
\end{example}

We now begin the proof of Theorem \ref{thm:poisson_linear}. For $\bm s \in V(G_n)_{|V(H)|}$ and an ordered subset $J \subseteq [|V(H)|]$ denote  by $\bm s_J=(s_j)_{j \in J}$, subset of indices $s_j$ such that $j \in J$. Then, define  
\begin{align}\label{eq:truncationset}
\mathcal A_\varepsilon(H, G_n) = \Bigg\{ \bm s \in V(G_n)_{|V(H)|}: & M_{J_1}(\bm s_{J_2}, H, G_n) \leq \varepsilon c_n^{|V(H)|-t}, \nonumber \\
& \text{ for all } J_1, J_2 \in V(H)_t, \text{ and all } t \in [2, |V(H)|-1]\Bigg\}. 
\end{align}
Informally, $\mathcal A_\varepsilon(H, G_n)$ counts the number of tuples $\bm s \in V(G_n)_{|V(H)|}$ such that the number of copies of $H$ passing through a subset of indices in $\bm s$ is `small'. 

\begin{example}(2-star)\label{ex:K12truncation} If $H=K_{1, 2}$ is the 2-star (with central vertex labeled $1$), then  $\mathcal A_\varepsilon(K_{1, 2}, G_n)$ consists of all $3$-tuples $\bm s=(s_1,s_2,s_3)$ of distinct vertices of $G_n$, such that,  
\begin{itemize}
\item[(1)] $a_{s_is_j}(G_n) (d_{G_n}(s_i)-a_{s_is_j}(G_n)) \leq \varepsilon c_n $  and $a_{s_is_j}(G_n) (d_{G_n}(s_j)-a_{s_is_j}(G_n)) \leq \varepsilon c_n $ (recall \eqref{eq:MK12}), that is, $\max\{d_{G_n}(s_i), d_{G_n}(s_j)\}=\varepsilon c_n + 1$ if there is an edge between $(s_i, s_j)$; and 
\item[(2)] $s_i$ and $s_j$ has at most $\varepsilon c_n$ common neighbors in $G_n$, 
\end{itemize}
for every $1 \leq i \ne j \leq 3$. 
\end{example}

Next, define the {\it main term}
\begin{align}\label{T1}
T^{+}_{\varepsilon}(H, G_n)=\frac{1}{|Aut(H)|}\sum_{\bm s \in \mathcal A_\varepsilon(H, G_n)} M(\bm s, H, G_n) \bm 1\{X_{=\bm s}\},
\end{align}
where $M(\bm s, H, G_n)= \prod_{(a,b) \in E(H)}a_{s_a s_b}(G_n)$, and the {\it remainder term} 
\begin{equation*}%\label{T2}
T^-_\varepsilon(H,G_n) = T(H,G_n) - T^+_\varepsilon(H,G_n).
\end{equation*}

\subsubsection{The Remainder Term}  We shall begin by showing that for each fixed $\varepsilon>0$, the remainder term $T^-_\varepsilon(H,G_n)$ converges in $L^1$ to $0$ as $n \rightarrow \infty$. Note that, $A \lesssim_\square B$ means $A\leq C \cdot B$, where $C:=C(\square) >0$ is a constant that depends only on the subscripted quantities. Similarly, $A{~}_{\square}{\gtrsim} B$ is $B \lesssim_\square A$.

\begin{lem}\label{lm:residual}
For each fixed $\varepsilon > 0$, $T^-_\varepsilon(H,G_n) \xrightarrow{L^1} 0$ as $n \rightarrow \infty$.
\end{lem}

\begin{proof}
To begin with, note that $$\E T^-_\varepsilon(H,G_n)=\frac{1}{c_n^{|V(H)|-1}}\sum_{\bm s \in V(G_n)_{|V(H)|} \backslash \mathcal A_\varepsilon(H, G_n)} \frac{M(\bm s, H, G_n)}{|Aut(H)|}.$$ 
Then, recalling the definition of $\cA_\varepsilon(H, G_n)$ from \eqref{eq:truncationset}, by an union bound
\begin{align}\label{eq:exp_remainder}
\E & T^-_\varepsilon(H,G_n) \nonumber \\
& \leq \frac{1}{c_n^{|V(H)|-1}} \sum_{t=2}^{|V(H)|-1} \sum_{\substack{J_1 \in V(H)_t \\ J_2 \in V(H)_t}}\sum_{\bm s \in V(G_n)_{|V(H)|} } \frac{M(\bm s, H, G_n)}{|Aut(H)|}  \bm 1\{ M_{J_1}(\bm s_{J_2}, H, G_n) > \varepsilon c_n^{|V(H)|-t} \} \nonumber \\
& \lesssim_H \frac{1}{c_n^{|V(H)|-1}} \sum_{t=2}^{|V(H)|-1} \sum_{\substack{J_1 \in V(H)_t \\ J_2 \in V(H)_t}} \sum_{\bm s \in V(G_n)_{|V(H)|} } M(\bm s, H, G_n)   \frac{M_{J_1}(\bm s_{J_2}, H, G_n)}{\varepsilon c_n^{|V(H)|-t}}. 
\end{align} 
In order to complete the proof, it thus suffices to show that 
\begin{equation}\label{eq:MJ1J2}
\sum_{\bm s \in V(G_n)_{|V(H)|}} M(\bm s, H, G_n) M_{J_1}(\bm s_{J_2}, H, G_n) = o(c_n^{2|V(H)|-t-1}), 
\end{equation}
for all $t\in [2,|V(H)|-1]$ and $J_1=(j_{11},\dots,j_{1t})$, $J_2 =(j_{21},\dots,j_{2t})\in V(H)_t$ (see Example \ref{ex:K12bound} for a special case). 

To this end, we have
\begin{align}
	\sum_{\bm s \in V(G_n)_{|V(H)|}} & M(\bm s, H, G_n) M_{J_1}(\bm s_{J_2}, H, G_n) \nonumber \\
	& =\sum_{\bm s_{J_2}} M_{J_2}(\bm s_{J_2}, H, G_n) M_{J_1}(\bm s_{J_2}, H, G_n) \nonumber \tag*{(summing over indices in $\bm s_{J_2^c}$)} \\
	& =\sum_{\bm r \in V(G_n)_{t}} M_{J_2}(\bm r, H, G_n) M_{J_1}(\bm r, H, G_n) \nonumber \tag*{(changing variable $\bm s_{J_2}$ to $\bm r$)} \\
	& =\sum_{\bm r \in V(G_n)_{t}} \big|\{(\phi,\psi)\in \hom_{\mathrm{inj}}(H,G_n)^2: \phi(j_{2a})=r_a=\psi(j_{1a}) \text{ for all } a \in [t]\}\big| \nonumber\\
	& = \big|\{(\phi,\psi)\in \hom_{\mathrm{inj}}(H,G_n)^2: \phi(j_{2,a})=\psi(j_{1,a}) \text{ for all } a \in [t]\}\big| \nonumber\\
	& \lesssim_H \sum_{t'=t}^{|V(H)|}\sum_{\substack{J_1'\supseteq J_1\\J_1'\in V(H)_{t'}}}\sum_{\substack{J_2'\supseteq J_2\\J_2'\in V(H)_{t'}}} N(H_{t'}(J_1',J_2'),G_n) \label{MM}
\end{align}  
The last step is based on the observation that a $(\phi,\psi) \in \hom_{\mathrm{inj}}(H,G_n)^2$ satisfying $\phi(j_{2a})=\psi(j_{1a}) \text{ for all } a \in [t]$, gives rise to a $t'$ join of $H$ with pivots $J_1'$ and $J_2'$ for some $t'\in [t,|V(H)|]$ and $J_1\subseteq J_1'\subseteq V(H)_t'$, $J_1\subseteq J_1'\subseteq V(H)_t'$ in at most finitely (depending only on $|V(H)|$) many ways. The reason we need to introduce $J_1'$ and $J_2'$, is that $\phi(j_2)$ may equal $\psi(j_1)$ for some $j_1 \notin J_1$ and $j_2 \notin J_2$. To elaborate, $J_2'$ consists of all those elements $j_2$ of $V(H)$, for which there exist an element $j_1$ of $V(H)$ such that $\phi(j_2) = \psi(j_1)$, and $J_1' = (\psi^{-1}(\phi(j_2)))_{j_2\in J_2'}$.

Now, note that the sum in \eqref{MM} is a finite sum (depending only on $H$). Further, for each $t' \in [t,|V(H)|-1], J_1\subseteq J_1'\in V(H)_{t'}$ and $J_2\subseteq J_2'\in V(H)_{t'}$, $$N(H_{t'}(J_1',J_2'),G_n) = o(c_n^{2|V(H)|-t'-1}) = o(c_n^{2|V(H)|-t-1}),$$ 
by assumption in Theorem \ref{thm:poisson_linear}. Lastly, for $J_1'\in V(H)_{|V(H)|}$ and $J_2'\in V(H)_{|V(H)|}$, $$N(H_{|V(H)|}(J_1',J_2'),G_n) = O(c_n^{|V(H)|-1})=o(c_n^{2|V(H)|-t-1}),$$

Therefore, $\lim_{n \rightarrow \infty}\E T^-_\varepsilon(H,G_n) \rightarrow 0$, completing the proof of the lemma. 
\end{proof}

\begin{example}\label{ex:K12bound}(2-star continued) To help the reader parse the above proof, we re-do the calculations for the case $H=K_{1, 2}$ (with central vertex labeled $1$), and $J_1=(2, 3)$ and $J_2=(1, 2)$. In this case, the LHS of \eqref{eq:MJ1J2} is 
\begin{align}
\sum_{(s_1, s_2, s_3) \in V(G_n)_{3}} M((s_1, s_2, s_3), & K_{1,2}, G_n) M_{(2, 3)}((s_{1}, s_2), K_{1,2}, G_n) \nonumber \\
&=\sum_{(s_1, s_2, s_3) \in V(G_n)_{3}} a_{s_1s_2}(G_n) a_{s_1 s_3}(G_n) M_{(2, 3)}((s_{1}, s_2), K_{1,2}, G_n) \nonumber \\
&\leq \sum_{s_1 \ne s_2  \in V(G_n)} a_{s_1s_2}(G_n) d_{G_n}(s_1) t_{G_n}(s_1, s_2) \nonumber\tag*{(recall Example \ref{ex:truncation})} \\
& \lesssim N(K_{3}, G_n)+N(\triangle_+, G_n) \nonumber ,
\end{align}
where $\triangle_+$ is the (3, 1)-tadpole (the graph obtained by joining a triangle and a single vertex with a bridge). Now,  $N(K_{3}, G_n)\lesssim N(K_{1,2}, G_n)=O(c_n^2)=o(c_n^3)$ and $N(\triangle_+, G_n)=o(c_n^3)$, by assumption in Theorem \ref{thm:poisson_linear}, which establishes \eqref{eq:MJ1J2}, for $H=K_{1,2}$, and $J_1=(2, 3)$, $J_2=(1, 2)$. 
\end{example}

\subsubsection{The Main Term: Moment Comparison}

To analyze $T_\varepsilon^{+}(H,G_n)$ we use the `independent approximation', where the indictors $\bm 1\{X_{=\bm s}\}$ are replaced by independent Bernoulli variables, for every subset of vertices in $G_n$ of size $|V(H)|$. To this end, define 
\begin{align}\label{eq:W}
J_\varepsilon^+(H, G_n)=\frac{1}{|Aut(H)|}\sum_{\bm s \in \mathcal A_\varepsilon(H, G_n)} M(\bm s, H, G_n) J_{\bar{\bm s}}, 
\end{align}
where  $\{J_S: S \subseteq V(G_n) \text{ and } |S|=|V(H)|\}$ is a collection of i.i.d. $\dBin(1,\frac{1}{c_n^{|V(H)|-1}})$ random variables.

\begin{lem}\label{lm:momentdiff} For every integer $r \geq 1$, 
$$\lim_{\varepsilon\to 0} \limsup_{n \to \infty}\left|\E T^+_\varepsilon(H,G_n)^r - \E J_\varepsilon^+(H, G_n)^r\right| = 0.$$ 
\end{lem}

\begin{proof} We begin with the following definition: 

\begin{defn}
Let $\sS_{\varepsilon, r, b}$ be the collection of all order $r$-tuples $\bm S=(\bm s_1, \bm s_2, \ldots, \bm s_r)$, where $\bm s_j=(s_{j1}, s_{j2}, \ldots, s_{j|V(H)|})$, for $j\in [r]$, such that 
\begin{itemize}
\item[--] $\bm s_j\in \mathcal A_{\varepsilon}(H,G_n)$, for all $j \in [r]$, 

\item[--] $M(\bm s_j, H, G_n) = 1$, for all $j \in [r]$.

\item[--] There are exactly $b$ distinct $|V(H)|$-element sets in the collection $\{\bar{\bm s}_1, \bar{\bm s}_2, \ldots, \bar{\bm s}_r\}$.
\end{itemize}

Finally, for a graph $F$, define 
\begin{align}\label{eq:Sset}
\sS_{\varepsilon, r, b}(F)=\{\bm S=(\bm s_1, \bm s_2, \ldots, \bm s_r) \in \sS_{\varepsilon, r, b} :  \cP(\bm S) \text{ is isomorphic to } F\},
\end{align} 
where $\cP(\bm S)=(V(\cP(\bm S)), E(\cP(\bm S)))$, such that 
\begin{align}\label{eq:Sgraph}
V(\cP(\bm S))=\bigcup_{j=1}^r \bar{\bm s}_j \quad \text{and} \quad  E(\cP(\bm S))=\bigcup_{j=1}^r \{(s_{ja}, s_{jb}): (a,b) \in E(H)\}.
\end{align}
\end{defn}

For $N \geq 1$, denote by $\sG_N$ the set of all labelled graphs on at most $N$ vertices. Moreover, let $\nu(F)$ denote the number of connected components of a graph $F$. 
Then by the multinomial expansion, 
\begin{align}\label{eq:momentdiffI}
|\E T_\varepsilon^+(H, G_n)^r -\E J_\varepsilon^+(H, G_n)^r| & \leq  \frac{1}{|Aut(H)|^r} \sum_{b=1}^r \sum_{\bm S \in \sS_{\varepsilon, r, b}}  \left|\E \prod_{t=1}^r  \bm 1\{X_{=\bm s_t}\}  - \E \prod_{t=1}^r  J_{\bm s_t} \right| \nonumber \\
& = \frac{1}{|Aut(H)|^r} \sum_{b=1}^r \sum_{\bm S \in \sS_{\varepsilon, r, b}} \left|\frac{1}{c_n^{|V(\cP(\bm S))|-\nu(\cP(\bm S))}}- \frac{1}{c_n^{b|V(H)|-b}} \right| \nonumber \\
& \lesssim_{H,r} \sum_{b=1}^r  \sum_{F \in \sG_{r|V(H)|}} \left|\frac{1}{c_n^{|V(F)|-\nu(F)}}- \frac{1}{c_n^{b|V(H)|-b}} \right| |\sS_{\varepsilon, r, b}(F)|, 
\end{align} 
Note that if the graph $F$ is connected and $\sS_{\varepsilon, r, b}(F)$ is non-empty, $|V(F)| -1 \leq  b|V(H)|-b$, and therefore, in general $|V(F)| -\nu(F) \leq  b|V(H)|-b$. Moreover, if $|V(F)|-\nu(F)= b |V(H)| -b$, the corresponding term in the sum in \eqref{eq:momentdiffI} is zero. This implies, 
\begin{align}\label{eq:momentdiffII}
|\E T_\varepsilon^+(H, G_n)^r & -\E J_\varepsilon^+(H, G_n)^r| \nonumber \\
& \lesssim_{H,r} \sum_{b=1}^r  \sum_{F \in \sG_{r|V(H)|}}  \frac{|\sS_{\varepsilon, r, b}(F)|}{c_n^{|V(F)|-\nu(F)}} \bm 1\{ |V(F)|-\nu(F)< b |V(H)| -b \}
\end{align}

To begin with assume that $F$ is connected and $|V(F)|-\nu(F)< b |V(H)| -b$. Then by Lemma \ref{lm:epscount}, $|\sS_{\varepsilon, r, b}(F)| \lesssim_{H,r} \varepsilon c_n^{|V(F)|-1}$. Next, if $F$ is disconnected with connected components $F_1, F_2, \ldots, F_{\nu(F)}$ such that $|V(F)| - \nu(F) < b|V(H)| - b$, then there exist $r_1, r_2, \ldots, r_{\nu(F)}$ and $b_1, b_2, \ldots, b_{\nu(F)}$, with  $\sum_{j=1}^{\nu(F)} r_j=r$ and $\sum_{j=1}^{\nu(F)} b_j=b$, such that $|V(F_i)| - \nu(F_i) \leq b_i|V(H)| - b_i$, for each $i \in [\nu(F)]$, with strict inequality for some $i \in [\nu(F)]$. More precisely, for the $i$-th connected component, $r_i$ is the number of tuples $\bm s_1, \bm s_2, \ldots, \bm s_{r_i}$ forming $F_i$, and $b_i$ is the number of distinct $|V(H)|$-element sets in the collection $\{\bar{\bm s}_1, \bar{\bm s}_2, \ldots, \bar{\bm s}_{r_i} \}$. Then using Lemma \ref{lm:epscount} below on each connected component gives $|\sS_{\varepsilon, r, b}(F)| \lesssim_{H,r} \varepsilon c_n^{|V(F)|-\nu(F)}$.  Therefore, every term in the sum in the RHS of \eqref{eq:momentdiffII} goes to zero as $n \rightarrow \infty$ followed $\varepsilon \rightarrow 0$. This completes the proof of the lemma, because the outside sum is finite (depending only on $H$ and $r$). 
\end{proof} 

\begin{lem}\label{lm:epscount} If $F$ is connected and $\sS_{\varepsilon, r, b}(F)$ is non-empty, then $|\sS_{\varepsilon, r, b}(F)| \lesssim_{H,r} c_n^{|V(F)|-1}$. Moreover, if $|V(F)| < b |V(H)|-b+1$, then $|\sS_{\varepsilon, r, b}(F)| \lesssim_{H,r} \varepsilon c_n^{|V(F)|-1}$.
\end{lem}

\begin{proof} To begin with assume that $|V(F)| < b |V(H)|-b+1$. Then without loss of generality, consider $\bm S=(\bm s_1, \bm s_2, \ldots, \bm s_r) \in \sS_{\varepsilon, r, b}(F)$ in the order given by Lemma \ref{connected}.  
For $ 0 \leq j \leq |V(H)|$, define 
\begin{equation}\label{addition}
\beta_j = \left|\left\{  t \in [2, r] : \left|\bar{\bm s}_t \bigcap \left(\bigcup_{a=1}^{t-1} \bar{\bm s}_a \right) \right| = |V(H)|-j \text{ and } \bar{\bm s}_t \notin \{\bar{\bm s}_1,\dots,\bar{\bm s}_{t-1}\} \right\}\right|. \end{equation}
The connectedness of $F$ and Lemma \ref{connected} implies that $\beta_{|V(H)|}=0$ and 
$\beta_1+ \ldots +\beta_{|V(H)|-2}\geq 1$. Note that 
$$|V(F)|=|V(H)|+ \sum_{j=1}^{|V(H)|-1} j \beta_j \quad \text{and} \quad b=1+\sum_{j=0}^{|V(H)|-1} \beta_j.$$

Now, define 
$$\sB=\left\{\bm \beta=(\beta_j)_{0\leq j\leq {|V(H)|}-1} \in \{0,1,...,r-1\}^{|V(H)|}: |V(H)| + \sum_{j=1}^{|V(H)|-1} j \beta_j = |V(F)| , \sum_{j=1}^{|V(H)|-2} \beta_j \geq 1\right\}.$$ Hence, for every $\varepsilon \in (0,1)$, using the fact that $\bm s_j \in \mathcal A_\varepsilon(H, G_n)$, for all $j \in [r]$, gives
\begin{align}
|\sS_{\varepsilon, r, b}(F)|  &\lesssim_{H, r}  \sum_{\bm \beta \in \sB} N(H,G_n)^{1+\beta_{|V(H)|-1}}\prod_{j=1}^{|V(H)|-2}(\varepsilon c_n^j)^{\beta_j} \nonumber \\
\label{eq:Fbd}&=   \sum_{\bm \beta \in \sB} N(H,G_n)^{1+\beta_{|V(H)|-1}} \varepsilon^{\sum_{j=1}^{|V(H)|-2} \beta_j} c_n^{\sum_{j=1}^{|V(H)|-2} j \beta_j} \\
& \lesssim_{H,r} \varepsilon \sum_{\bm \beta \in \sB} c_n^{(1+\beta_{|V(H)|-1})(|V(H)|-1)}  c_n^{\sum_{j=1}^{|V(H)|-2} j \beta_j} \tag*{(using $\sum_{j=1}^{|V(H)|-2} \beta_j \geq 1$ and $N(H, G_n)=O(c_n^{|V(H)|-1})$)}  \nonumber \\
&=  \varepsilon c_n^{|V(F)|-1} |\sB| \lesssim_{H, r} \varepsilon c_n^{|V(F)|-1}, \nonumber  
\end{align}    
where the last step uses the crude estimate $|\sB| \leq r^{|V(H)|}$. See Example \ref{ex:K12beta12} for an illustration of the argument in the above display in a special case.

Finally, suppose that $|V(F)| \leq b |V(H)|-b+1$. Since $F$ is connected and $\sS_{\varepsilon, r, b}(F)$ is non-empty, there exists $\bm S=(\bm s_1, \bm s_2, \ldots, \bm s_r) \in \sS_{\varepsilon, r, b}(F)$ such that $\beta_{|V(H)|} = 0$, where $(\beta_0, \ldots ,\beta_{|V(H)|})$ is defined as in \eqref{addition}. Define 
$$\sB'=\left\{\bm \beta=(\beta_j)_{0\leq j\leq {|V(H)|}-1} \in \{0,1,...,r-1\}^{|V(H)|}: |V(H)| + \sum_{j=1}^{|V(H)|-1} j \beta_j = |V(F)| \right\}.$$ For a $\bm \beta \in \sB'$, $\sum_{j=1}^{|V(H)|-2} \beta_j $ can be zero, but using $\varepsilon =1$  in \eqref{eq:Fbd} (with $\sB$ replaced by $\sB'$)  
$$|\sS_{\varepsilon, r, b}(F)| \lesssim_{H,r}   \sum_{\bm \beta \in \sB'} N(H,G_n)^{1+\beta_{|V(H)|-1}}  c_n^{\sum_{j=1}^{|V(H)|-2} j \beta_j}  \lesssim_{H, r}   c_n^{|V(F)|-1},$$ completing the proof of the lemma.  
\end{proof}

\begin{example}\label{ex:K12beta12}(2-star continued) Suppose $H=K_{1, 2}$, and $F=\mathcal P(\bm S)$ is connected, where $\bm S=(\bm s_1, \bm s_2, \ldots, \bm s_r)$.  If at the $j$-th step a single new vertex is added, then the number of ways to choose such a triple $\bm s_j$ from $\cA_{\varepsilon}(K_{1, 2}, G_n)$ is at most $O(\varepsilon c_n)$ (recall Example \ref{ex:K12truncation}). On the other hand, if two new vertices are added, the number of possible triples is trivially bounded by $O(N(K_{1, 2}, G_n))$. This implies the bound in \eqref{eq:Fbd} because, the number of times 1 or 2 vertices are added in the sequence $\bm S$ is $\beta_1$ and $\beta_2$, respectively (note that 3 vertices are always added at the first step, which contributes the extra factor of $O(N(K_{1, 2}, G_n))$). 
\end{example}

\subsubsection{Completing the Proof of Theorem \ref{thm:poisson_linear}} Lemma \ref{lm:momentdiff} shows the moments of $T_\varepsilon^{+}(H, G_n)$ and $J_\varepsilon^{+}(H, G_n)$ are asymptotically close. Now, we derive the limiting distribution of $J_\varepsilon^{+}(H, G_n)$.

\begin{lem}\label{lm:Wpoisson} Let $J_\varepsilon^{+}(H, G_n)$ be as defined in \eqref{eq:W}.  Then for every $\varepsilon > 0$,  as $n \rightarrow \infty$, 
$$J_\varepsilon^{+}(H, G_n) \rightarrow \sum_{k=1}^{N(H,K_{|V(H)|})} k Z_k$$ 
in distribution and in moments, where $Z_k\sim \dPois(\lambda_k)$ and the collection $\{Z_k: 1\leq k\leq N(H,K_{|V(H)|})\}$ is independent.
\end{lem}
\begin{proof} For each $k \in [1,N(H,K_{|V(H)|})]$, define 
\begin{align}\label{eq:ck}
\cD_k(H, G_n) = \{S\subseteq V(G_n): |S| = |V(H)| \text{ and } N(H,G_n[S]) = k\},
\end{align} 
where $G_n[S]$ is the subgraph of $G_n$ induced on the set $S$.\footnote{For example, $H=K_{1,2}$, then $\cD_1(K_{1, 2},G_n)$ is the collection of all induces 2-stars in $G_n$, $\cD_2(K_{1, 2}, G_n)$ is empty, and $\cD_3(K_{1, 2}, G_n)$ is the number of induced triangles in $G_n$.}

 For every subset $S := \{s_1,\dots,s_{|V(H)|}\}$ of  $V(G_n)$ of size $|V(H)|$, let $\sigma_0(S)=(s_{\sigma_0(1)}, s_{\sigma_0(2)}, \ldots, s_{\sigma_0(|V(H)|)})  \in V(G_n)_{|V(H)|}$, be such that $s_{\sigma_0(1)} <  s_{\sigma_0(2)}  < \cdots < s_{\sigma_0(|V(H)|)}$. Now, define
\begin{align*}
\sB_\varepsilon(H, G_n) = \Bigg\{S \subseteq V(G_n): |S|=|V(H)| \text{ and } \sigma_0(S) \in  \mathcal A_\varepsilon(H, G_n)\Bigg\}.
\end{align*}
Then recalling the definition of $J_\varepsilon^{+}(H, G_n)$ from \eqref{eq:W}, we have 
\begin{align}%\label{eq:Jsum}
J_\varepsilon^{+}(H, G_n)&=\frac{1}{|Aut(H)|}\sum_{\bm s \in \mathcal A_\varepsilon(H, G_n)} M(\bm s, H, G_n) J_{\bar{\bm s}} \nonumber \\
&= \frac{1}{|Aut(H)|}\sum_{k=1}^{N(H,K_{|V(H)|})} \sum_{\substack{\bm s\in \cA_\varepsilon(H, G_n)\\\bar{\bm s}\in \cD_k(H, G_n)}}  M(\bm s, H, G_n) J_{\bar{\bm s}} \nonumber \\ 
&= \frac{1}{|Aut(H)|}\sum_{k=1}^{N(H,K_{|V(H)|})} \sum_{S \in \cD_k(H, G_n)\bigcap \sB_\varepsilon(H, G_n)} \sum_{\substack{\bm s \in V(G_n)_{|V(H)|}\\\bar{\bm{s}}=S}} M(\bm s, H, G_n) J_{\bar{\bm s}} \nonumber\\ &= \sum_{k=1}^{N(H,K_{|V(H)|})} \sum_{S \in \cD_k(H, G_n)\bigcap \sB_\varepsilon(H, G_n)}N(H,G_n[S]) J_S \nonumber\\ &= \sum_{k=1}^{N(H,K_{|V(H)|})} k \sum_{S \in \cD_k(H, G_n)\bigcap \sB_\varepsilon(H, G_n)} J_S \nonumber. 
\end{align}
Now, note that, by definition, the collection $\{\sum_{S \in \cD_k(H, G_n)\bigcap \sB_\varepsilon(H, G_n)} J_S: 1\leq k\leq N(H,K_{|V(H)|})\}$ is independent, and for every fixed $k \in [1,N(H,K_{|V(H)|})]$, $$J_{n,\varepsilon}(k) := \sum_{S \in \cD_k(H, G_n)\bigcap \sB_\varepsilon(H, G_n)} J_S \sim \dBin\left(\left|\cD_k(H, G_n)\bigcap \sB_\varepsilon(H, G_n)\right|, \frac{1}{c_n^{|V(H)|-1}}\right).$$ Therefore, to prove the lemma it suffices to show that $\E J_{n,\varepsilon}(k) \rightarrow \lambda_k$, for every $k \in [1,N(H,K_{|V(H)|})]$. To this end, note that
\begin{align}
& \left|\cD_k(H, G_n)\right| - \left|\cD_k(H, G_n)\bigcap \sB_\varepsilon(H, G_n)\right| \nonumber \\
&= \left|\{S\in \cD_k(H, G_n): \sigma_0(S) \notin \mathcal A_\varepsilon(H, G_n)\}\right| \nonumber \\ &\leq \frac{1}{|Aut(H)|}\left|\{\bm s \in V(G_n)_{|V(H)|} : M(\bm s , H, G_n) = 1 \text{ and } \bm s \notin \mathcal A_\varepsilon(H, G_n)\}\right|\nonumber \\ &= \frac{1}{|Aut(H)|}\sum_{\bm s \in V(G_n)_{|V(H)|}} M(\bm s , H, G_n) \bm 1\{\bm s \notin \mathcal A_\varepsilon(H, G_n)\}\nonumber \\&= c_n^{|V(H)|-1}\E T_\varepsilon^-(H,G_n)\nonumber\\ &= o(c_n^{|V(H)|-1})\nonumber
\end{align}
by Lemma \ref{lm:residual}. Thus, $\E J_{n,\varepsilon}(k) = \frac{1}{c_n^{|V(H)|-1}} |\cD_k(H, G_n)| + o(1)$. The lemma now follows from assumption \ref{eq:Finduced} of Theorem \ref{thm:poisson_linear}, and the observation that $|\cD_k(H, G_n)| = \sum_{F\in \sC_{H,k}} N_{\mathrm{ind}}(F,G_n)$. 
\end{proof}

\begin{example}(2-star continued) If $H=K_{1,2}$, then every set $S \in \sB_\varepsilon(K_{1,2}, G_n)$ for which the induced graph $G_n[S]$ is a triangle, contributes to $J_{\varepsilon}^+(K_{1,2},G_n)$ the same Bernoulli variable three times, since $N(K_{1, 2}, K_3)=3$. On the other hand, if the induced graph $G_n[S]$ is a $2$-star, then $S$ contributes a single Bernoulli variable to $J_{\varepsilon}^+(K_{1,2},G_n)$. By the joint independence of the collection $J_S$ over all three-element subsets $S$ of $V(G_n)$, it follows that $J_\varepsilon^+(H,G_n)=J_{n,\varepsilon}'+3 J_{n,\varepsilon}''$, where $J_{n,\varepsilon}'$ and $J_{n,\epsilon}''$ are independent Binomial random variables. The calculation in the above lemma implies that $\E J_{n,\varepsilon}'=\frac{1}{c_n^2} |\cD_1(K_{1, 2}, G_n)|+o(1)=\frac{1}{c_n^2} N_{\mathrm{ind}}(K_{1, 2}, G_n)+o(1)=\lambda_1+o(1)$, and, similarly, $\E J_{n,\varepsilon}''=\frac{1}{c_n^2} |\cD_3(K_{1, 2}, G_n)|+o(1)=\frac{1}{c_n^2} N_{\mathrm{ind}}(K_{3}, G_n)+o(1)= \lambda_3+o(1)$ (by assumption \eqref{eq:Finduced}). 
\end{example}

To complete the proof of Theorem \ref{thm:poisson_linear}, let $Z$ be the random variable on the RHS of \eqref{eq:Tlinear}. The above lemma, combined with Lemma \ref{lm:momentdiff}, implies that, for all $r \geq 1$,
\begin{align}\label{eq:T_moments}
\lim_{\varepsilon \rightarrow 0} \limsup_{n \rightarrow \infty}|\E(T_\varepsilon^{+}(H, G_n))^r- \E Z^r|=0.
\end{align} 
Furthermore, the random variable $Z$ has a finite moment generating function, which implies, by Lemma \ref{lm:characteristic_function}, that $T_\varepsilon^{+}(H, G_n) \dto Z$, as $n \rightarrow \infty$ followed by $\varepsilon \rightarrow 0$. Hence, $T(H, G_n) \dto Z$ (by Lemma \ref{lm:residual}), as $n \rightarrow \infty$, completing the proof. \hfill $\Box$ \\

Recently, Bhattacharya and Mukherjee \cite{BMdense} characterized the limiting distribution of $T(H, G_n)$, when $G_n$ is a sequence of dense graphs converging to a graphon $W$. In the following remark, we discuss how Theorem \ref{thm:poisson_linear} can be used to re-derive \cite[Theorem 1.1]{BMdense}, which obtains the limiting distribution $T(H, G_n)$ for a converging sequence of dense graphs, in the Poisson regime.

\begin{remark}\label{rm:dense}(Dense Graphs) Recall that a {\it graphon} $W:[0, 1]^2 \rightarrow [0, 1]$  is a measurable function satisfying $W(x, y) = W(y,x)$, for all $x, y$. A finite simple graph $G=(V(G), E(G))$ can also be represented as a graphon in a natural way: Define $f^G(x, y) =\boldsymbol 1\{(\ceil{|V(G)|x}, \ceil{|V(G)|y})\in E(G)\}$, that is, partition $[0, 1]^2$ into $|V(G)|^2$ squares of side length $1/|V(G)|$, and let $f^G(x, y)=1$ in the $(i, j)$-th square if $(i, j)\in E(G)$, and 0 otherwise. For a simple graph $F$ with $V (F)= \{1, 2, \ldots, |V(F)|\}$, define $$t(F,W) =\int_{[0,1]^{|V(F)|}}\prod_{(i,j)\in E(F)}W(x_i,x_j) \mathrm dx_1\mathrm dx_2\cdots \mathrm dx_{|V(F)|}$$ (continuous analogue of the homomorphism density). The basic definition of graph-limit theory is the following: A sequence of graphs $\{G_n\}_{n\geq 1}$ is said to {\it converge to $W$} if for every finite simple graph $F$, 
$\lim_{n\rightarrow \infty}t(F, G_n) = t(F, W)$ (refer to Lov\'asz \cite{lovasz_book} for more on graph limit theory).

In \cite[Theorem 1.1]{BMdense} the authors showed that $T(H, G_n)$ converges to a linear combination of independent Poisson random variables, whenever $\E(T(H, G_n))=O(1)$, and $G_n$ converges to a graphon $W$ such that $t(H, W) > 0$. This result can de derived as a  consequence of Theorem \ref{thm:poisson_linear} as follows: If $G_n$ is a sequence of dense graphs, as above, colored with $c_n$ colors such that $\E(T(H, G_n)) \rightarrow \lambda$, then 
$$c_n=\Theta(|V(G_n)|^{\frac{|V(H)|}{|V(H)|-1}}),$$ 
since $N(H, G_n)=\Theta(|V(G_n)|^{|V(H)|})$,\footnote{For two non-negative sequences $(a_n)_{n \geq 1}$ and $(b_n)_{n \geq 1}$, $a_n=\Theta(b_n)$ means that there exist positive constants $C_1, C_2 $, such that $C_1 b_n \leq a_n \leq C_2 b_n$, for all $n$ large enough.} by assumption $t(H, W)>0$. Therefore, for $t \in [2, |V(H)|-1]$, and $F \in \sJ_t(H)$, 
\begin{align} 
N(F, G_n)=O(|V(G_n)|^{|V(F)|})=O(|V(G_n)|^{2|V(H)|-t})&= O\left(c_n^{\frac{2|V(H)|-t}{2}}\right) \nonumber \\
&=o(c_n^{2|V(H)|-t-1}), \nonumber
\end{align} 
which establishes the second assumption of Theorem \ref{thm:poisson_linear}. Finally, since the convergence of $G_n$ to a graphon $W$ implies the convergence of the proportion of induced subgraphs in $G_n$,   the limits in \eqref{eq:Finduced} exist, and, hence, \cite[Theorem 1.1]{BMdense} follows: $$T(H, G_n) \dto \sum_{F \supseteq H: |V(F)|=|V(H)|} N(H, F) X_F,$$ where $X_F\sim \dPois(\lambda_F)$ (where $\lambda_F:=\lim_{n\rightarrow \infty}\frac{1}{c_n^{|V(F)|-1}} N_{\mathrm{ind}}(F, G_n)$ exists  because of the convergence of $G_n$) and the collection $\{X_F: F \supseteq H \text{ and } |V(F)|=|V(H)|\}$ is independent. As usual, we consider only non-isomorphic (unlabelled) super-graphs $F$ of $H$, whenever we write $F \supseteq H$.
\end{remark}

\subsection{Proof of Theorem \ref{thm:poisson}}
\label{sec:pfpoisson}

Note that, for  $\bm s_1, \bm s_2 \in V(G_n)_{|V(H)|}$ such that $\bar{\bm s}_1\bigcap \bar{\bm s}_2 \neq \phi$, 
$$\Cov (\bm 1\{X_{=\bm s_1}\}, \bm 1\{X_{=\bm s_2} \})=\frac{1}{c_n^{2|V(H)|-|\bar{\bm s}_1 \bigcap \bar{\bm s}_2|-1}}-\frac{1}{c_n^{2|V(H)|-2}} .$$ The covariance is $0$ if $\bar{\bm s}_1\bigcap \bar{\bm s}_2$ is empty or singleton. Therefore,

\begin{align}\label{eq:varT}
\Var T(H, G_n)=& R_{1,n}+R_{2,n}
\end{align}
where 
\begin{align}\label{eq:R1}
R_{1,n} =\frac{1}{c_n^{|V(H)|-1}}\left(1-\frac{1}{c_n^{|V(H)|-1}} \right) N(H, G_n) \rightarrow \lambda,
\end{align} 
since $\E T(H, G_n) =\frac{1}{c_n^{|V(H)|-1}} N(H, G_n) \rightarrow \lambda$,  and the covariance terms 
\begin{align}\label{eq:R2I}
R_{2,n} =\sum_ {t=2}^{|V(H)|} \frac{1}{c_n^{2|V(H)|-t-1}}\left(1-\frac{1}{c_n^{t-1}}\right)  |\mathcal{K}(t,H,G_n)|,  
\end{align} 
where $\mathcal{K}(t,H,G_n)$ is the set of all ordered pairs $(H_1,H_2)$ such that $H_1 \neq H_2$ are subgraphs of $G_n$ isomorphic to $H$, sharing exactly $t$ vertices in common.
Now, the assumption $\Var T(H, G_n) \rightarrow \lambda$ and \eqref{eq:R1} implies that $R_{2,n} \rightarrow 0$. Therefore, 
\begin{equation}\label{R2new}
|\mathcal{K}(t,H,G_n)| = o(c_n^{2|V(H)|-t-1}) .
\end{equation}
for every $t \in [2,|V(H)|]$. Further, for every $t \in [2,|V(H)|-1]$, 
\begin{equation}\label{R2new1}
\sum_{F \in \sJ_t(H)} N(F,G_n) \lesssim_H  |\mathcal{K}(t,H,G_n)|.
\end{equation}
Combining \eqref{R2new} and \eqref{R2new1} imply that $N(F,G_n) = o(c_n^{2|V(H)|-t-1})$ for all $t\in [2,|V(H)|-1]$ and all $F \in \sJ_t(H)$. 

Next consider $t=|V(H)|$ in \eqref{R2new} and note that 
\begin{equation}\label{R2new2}
|\mathcal{K}(|V(H)|,H,G_n)| = \sum_{k=2}^{N(H,K_{|V(H)|})} k(k-1)|\cD_k(H, G_n)|, 
\end{equation}
where $\cD_k(H, G_n)$ is as defined in \eqref{eq:ck}. This follows by first choosing the common vertex set from exactly one of the collections $\cD_k(H, G_n)$ for $k \in [2,N(H,K_{|V(H)|})]$, and then choosing the pair $(H_1,H_2)$ in $k(k-1)$ ways.\footnote{For example, if $H=C_4$ is the 4-cycle, and $G_n=K_n$ is the complete graph, the LHS in \eqref{R2new} is $6 {n \choose 4}$ (choose $H_1$ from $G_n$ in $N(C_4, G_n)=3 {n \choose 4}$ ways, which leaves 2 choices for $H_2$) which matches with the RHS, since $|\cD_2(C_4, G_n)|=0$, and $|\cD_3(C_4, G_n)|= {n \choose 4}$ (every 4-tuple in $G_n$ has an induced $K_4$ and $N(C_4, K_4)=3$).}

Combining \eqref{R2new} and \eqref{R2new2} gives $|\cD_k(H, G_n)| = o(c_n^{|V(H)|-1})$ for all $k \in [2,N(H,K_{|V(H)|})]$. Now, using $|\cD_k(H, G_n)|=\sum_{F\in \sC_{H,k}} N_{\mathrm{ind}}(F,G_n)$ gives 
\begin{equation}\label{R2new3}
\sum_{F\in \sC_{H,k}} N_{\mathrm{ind}}(F,G_n) = o(c_n^{|V(H)|-1}),
\end{equation}            
that is, $\lambda_k = 0$ for all $k \in [2,N(H,K_{|V(H)|})]$. Lastly, by a counting argument similar to the one used above,  
\begin{equation}\label{R2new4}
N(H,G_n) = \sum_{k=1}^{N(H,K_{|V(H)|})} k|\cD_k(H, G_n)| .
\end{equation}
Since $\frac{1}{c_n^{|V(H)|-1}}N(H,G_n) \rightarrow \lambda$, \eqref{R2new4} now implies that $\frac{1}{c_n^{|V(H)|-1}}|\mathcal{D}_1(H,G_n)|\rightarrow \lambda$, and hence,
\begin{equation*}%\label{R2new5}
\frac{\sum_{F\in \sC_{H,1}} N_{\mathrm{ind}}(F,G_n)}{c_n^{|V(H)|-1}} \rightarrow \lambda.
\end{equation*}
Condition \eqref{eq:Finduced} of Theorem \ref{thm:poisson_linear} is thus satisfied with $\lambda_1 = \lambda$ and $\lambda_k = 0$ for all $k \in [2,N(H,K_{|V(H)|})]$. Theorem \ref{thm:poisson_linear} now implies that $T(H,G_n) \dto \dPois(\lambda)$, completing the proof of the second-moment phenomenon for monochromatic subgraphs.

\section{Proof of Theorem \ref{thm:star}}
\label{sec:pfstar}

The if part follows directly from Theorem \ref{thm:poisson}. The proof of the only-if part is given in Section \ref{sec:pfonlyif}. The counter-example when $H$ is not a star-graph is explained in Section \ref{sec:Hexample}.

\subsection{$T(K_{1, r}, G_n) \dto \dPois(\lambda)$ implies Convergence of Moments}
\label{sec:pfonlyif}

We begin by showing that $T(K_{1, r}, G_n) \dto \dPois(\lambda)$ implies  $\E T(K_{1, r}, G_n)$ is bounded. 

\begin{lem}\label{lm:rstar}
Let $\{G_n\}_{n\ge 1}$ be a sequence of deterministic graphs colored uniformly with $c_n$ colors. Then 
$$T(K_{1,r},G_n)\pto 
\left\{\begin{array}{ccc}
0  & \text{if}  &  \lim_{n\rightarrow\infty} \E T(K_{1, r}, G_n)=0,\\
\infty  & \text{if}  & \lim_{n\rightarrow\infty} \E T(K_{1, r}, G_n)=\infty.
\end{array}\right.
$$
\end{lem} 

\begin{proof}
If $\E T(K_{1,r},G_n)\rightarrow 0$, then $\P(T(K_{1,r},G_n)>0)\le \E T(K_{1,r},G_n) \rightarrow 0$. 

To show that $T(K_{1,r},G_n)$ diverges, if $\E T(K_{1,r},G_n)\rightarrow \infty$, it suffices to show that $\Var T(K_{1,r},G_n)=o((\E T(K_{1,r},G_n))^2)$, because this implies that $T(K_{1,r},G_n)/\E (T(K_{1,r},G_n))\pto 1$, which is possible only if $T(K_{1,r},G_n)\pto \infty$.  

Write $\Var T(K_{1,r},G_n) = R_{1,n} + R_{2,n}$, as in \eqref{eq:varT} (with $H = K_{1,r}$). Clearly, 
\begin{equation}\label{R1bound}
R_{1,n} \leq \E T(K_{1,r}(G_n)) = o((\E T(K_{1,r},G_n))^2) .
\end{equation}
Next, observe that for each $t \in [2, r+1]$, 
\begin{equation}\label{nonunique}
\sum_{\substack{\bm u \ne \bm v \in V(G_n)_{|V(H)|}\\  |\bar{\bm u} \bigcap \bar{\bm v}|=t}} \frac{M(\bm u, K_{1,r}, G_n) M(\bm v, K_{1,r}, G_n)}{|Aut(K_{1,r})|^2} \lesssim_r \sum_{F \in \sJ_t(K_{1,r})} N(F,G_n) .
\end{equation}
For each $F \in \sJ_t(K_{1,r})$, by Lemma \ref{lm:rstarcount}, $N(F,G_n) \lesssim_r N(K_{1,r},G_n)^\frac{2r-t+1}{r}$. Therefore, by \eqref{eq:R2I} and \eqref{nonunique}, 
\begin{align}\label{R2bound}
	R_{2,n} \lesssim_r \sum_ {t=2}^{|V(H)|}  \frac{N(K_{1,r},G_n)^\frac{2r-t+1}{r}}{c_n^{2r-t+1}} = \sum_ {t=2}^{|V(H)|}  \left(\E T(K_{1,r},G_n)\right)^{2-\frac{t-1}{r}} = o((\E T(K_{1, r}, G_n))^2).
\end{align} 
Now, \eqref{R1bound} and \eqref{R2bound} imply that $\Var T(K_{1,r},G_n)=o((\E T(K_{1,r},G_n))^2)$, completing the proof of the lemma.
\end{proof}

By the above proposition, $T(K_{1, r}, G_n) \dto \dPois(\lambda)$, implies that $\E T(K_{1, r}, G_n) =\frac{N(K_{1, r}, G_n)}{c_n^r}=\Theta(1)$. Therefore, by Lemma \ref{lm:rstarcount}, 
\begin{align}\label{eq:Fcountstar}
N(F, G_n)=O(c_n^{|V(F)|-\nu(F)}),
\end{align} for any graph $F$ which is the union of $r$-stars with $\nu(F)$ connected components. Using this we can show that the moments of $T(K_{1, r}, G_n)$ are bounded. To this end, set $r'=r+1$ and fix an integer $m \geq 1$. Let $\cS$ be the collection of all ordered $m$-tuples $(\bm s_1, \bm s_2, \ldots, \bm s_m)$, where $\bm s_j := (s_{j1},\dots,s_{jr'}) \in V(G_n)_{r'}$, for $j \in [m]$, and $M(\bm s_j,K_{1,r},G_n)=1$, for every $j \in [m]$. Then by the multinomial expansion, 
\begin{align}\label{eq:starmoment}
\E T(K_{1, r}, G_n)^m  =  \frac{1}{|Aut(K_{1, r})|^m}\sum_{\cS}  \E \prod_{j=1}^m \bm 1\{X_{=\bm s_j}\}  & = \frac{1}{|Aut(K_{1, r})|^m} \sum_{\cS}  \frac{1}{c_n^{|V(F)|-\nu(F)}}, 
\end{align}
where $F=F(\bm s_1,\cdots,\bm s_m)$ is the graph on vertex set $V(F)=\bigcup_{j=1}^m \bar{\bm s}_j$ and edge set $\bigcup_{j=1}^m \{(s_{ja}, s_{jb}): (a,b) \in E(K_{1, r})\}$, and $\nu(F)$ is the number of connected components of $F$. 
Denote by $\sG_m(K_{1,r})$ the collection of all unlabelled graphs formed by the join of $m$ isomorphic copies of $K_{1, r}$.\footnote{For any graph $H$, $\sG_2(H)$ is the collection of all non-isomorphic graphs obtained the join of 2 copies of $H$, as in Definition \ref{defn:tjoin}. For $m \geq 3$,  define $\sG_m(H)$ inductively, as the collection of all non-isomorphic graphs $F$, that can be obtained by identifying $t$ vertices of $H$, for some $t \in [1, |V(H)|]$, with $t$ vertices of some graph $F_1 \in \sG_{m-1}(H)$.}

Then \eqref{eq:starmoment} implies    
\begin{align*}
\E T(K_{1,r},G_n)^m \lesssim_{r, m} \sum_{F\in \cH_{r, m}}\frac{N(F,G_n)}{c_n^{|V(F)|-\nu(F)}}=O(1),
\end{align*}
using \eqref{eq:Fcountstar}, since $\cH_{r, m}$ is a finite set (depending only on $r$ and $m$). This implies, by uniform integrability, 
$\E T(K_{1, r}, G_n)^m \rightarrow \E(\dPois(\lambda))^m$, for every $m \geq 1$. In particular,  $\E T(K_{1, r}, G_n) \rightarrow \lambda$ and $\Var T(K_{1, r}, G_n) \rightarrow \lambda$, as required in \eqref{eq:rcondition}. Therefore, to complete the proof of the only if part it remains to prove the following lemma:  

\begin{lem}\label{lm:rstarcount}
Let $F$ be a graph formed by the union of $r$-stars with $\nu(F)$ connected components. Then for any graph $G_n$
\begin{equation*}%\label{eq:count1}
N(F,G_n)\lesssim_{F,r} N(K_{1, r}, G_n)^{\frac{|V(F)|-\nu(F)}{r}}.
\end{equation*} 
\end{lem}

\begin{proof}
Let $F_1, F_2, \dots , F_{\nu(F)}$ denote the connected components of $F$. Clearly, $F_i$ contains an $r$-star for each $1\leq a \leq \nu(F)$. Hence, for every $1\leq a \leq \nu(F)$, 
\begin{equation}\label{eq:count2}
N(F_a,G_n) \lesssim_{F_a,r} N(K_{1,r}, G_n)\left(\Delta(G_n)\right)^{|V(F_a)|-r-1}, 
\end{equation}
where $\Delta(G_n)$ is the maximum degree in $G_n$. On the other hand, 
$$N(K_{1,r},G_n) = \sum_{v\in V(G_n)} \binom{d_v}{r}\geq \binom{\Delta(G_n)}{r} {~}_{r}\gtrsim \Delta(G_n)^r.$$
This implies that $(\Delta(G_n))^{|V(F_a)|-r-1} \lesssim_r N(K_{1,r},G_n)^{\frac{|V(F_a)|-1}{r}-1}$, and from  \eqref{eq:count2}, 
\begin{equation}\label{percomponent}
N(F_a,G_n) \lesssim_{F_a,r} N(K_{1,r},G_n)^{\frac{|V(F_a)|-1}{r}}.
\end{equation}
Since \eqref{percomponent} is true for every $1\leq a \leq \nu(F)$, 
$$N(F,G_n) \leq \prod_{a=1}^{\nu(F)} N(F_a,G_n) \lesssim_{F,r} N(K_{1,r},G_n)^{\sum_{a=1}^{\nu(F)}\frac{|V(F_a)|-1}{r}} = N(K_{1,r},G_n)^{\frac{|V(F)|-\nu(F)}{r}},$$
completing the proof. 
\end{proof}

\subsection{Counterexample when $H$ is not a star-graph}\label{sec:Hexample} In this section, we construct a graph sequence $G_n(H)$ such that $T(H, G_n(H)) \dto \dPois(\lambda)$, but \eqref{eq:expvar} does not hold, whenever $H$ is connected and is not a star-graph.

%%%%%%%%%%%%%%%%%%%%%%%%%%%%%%%%%%%%%%%%%%%%%%%%%%%%%%%%%%%%%%%%%%%%%%%%%%%%%%%%
\begin{figure}[h]
\centering
\begin{minipage}[l]{1.0\textwidth}
\centering
\includegraphics[width=3.25in]
    {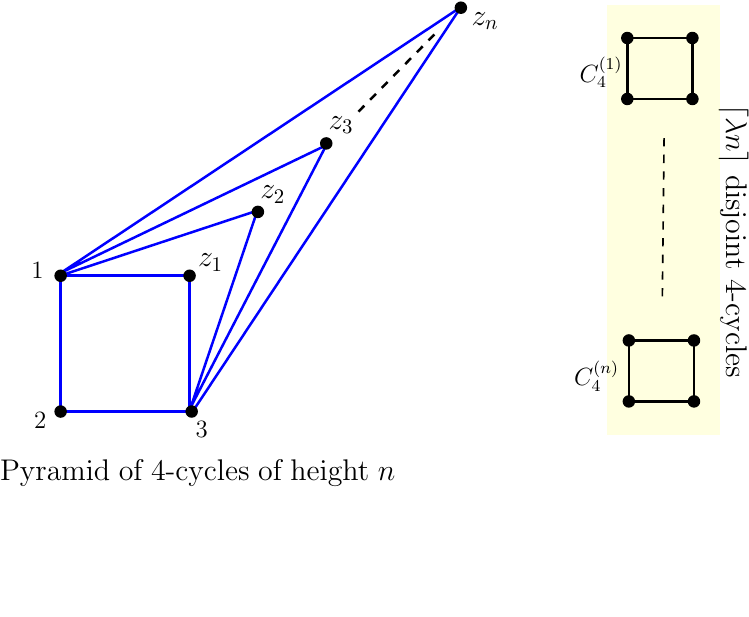}\\
\vspace{-0.5in}
%\small{(a)}
\end{minipage}
%\begin{minipage}[r]{0.49\textwidth}
%\centering
%\includegraphics[width=2.75in]
%    {Triangle_Bipartite}\\
%%\small{(b)}
%\end{minipage}
\caption{\small{Illustration showing Poisson convergence does not imply convergence of moments, when $H=C_4$.}}
\label{fig:example_moment}
\end{figure}
%%%%%%%%%%%%%%%%%%%%%%%%%%%%%%%%%%%%%%%%%%%%%%%%%%%%%%%%%%%%%%%%%%%%%%%%%%%%%%%%%%%%%%%%%%%%%%%%%%%%%%%%%%

\begin{defn} Fix an integer $n \geq 1$. Let $H_1, H_2, \ldots, H_n$ be isomorphic copies of $H$, with  $V(H)=\{1, 2, \ldots, |V(H)|-1, |V(H)|\}$ and $H_a=(V(H_a), E(H_a))$, such that $V(H_a)=\{1, 2, \ldots, |V(H)|-1, z_a\}$, where $\phi(v)=v$, for $v \in [|V(H)|-1]$, and $\phi(|V(H)|)=z_a$, is an isomorphism of $H$ and $H_a$, for $a \in [n]$. Define {\it the pyramid of $H$ of height $n$} as follows: 
$$\cP_n(H)=\left(\bigcup_{a=1}^n V(H_a), \bigcup_{a=1}^n  E(H_a) \right).$$ 
\end{defn}

Let $G_n(H)$ be the disjoint union of  $\cP_n(H)$ and $\ceil{\lambda n}$ disjoint copies of $H$. (Figure \ref{fig:example_moment} illustrates this construction when $H=C_4$ is the 4-cycle). 

\begin{lem}\label{lm:Hmoment} Suppose $H$ is connected and is not a star-graph. Let $\cP_n(H)$ be a pyramid of $H$ of height $n$, as defined above. Then every copy of $H$ in $\cP_n(H)$ passes through at least two vertices in $\{1, 2, \ldots, |V(H)|-1\}$.
\end{lem}

\begin{proof} Since $\{z_1, z_2, \ldots, z_n\}$ is an independent set, by construction, and $H$ is connected, every copy of $H$ in $\cP_n(H)$ must pass through at least $1$ vertex in $\{1, 2, \ldots, |V(H)|-1\}$. Suppose there exists a copy of $H$ in $\cP_n(H)$ which passes through exactly $1$ vertex (say $k$) in $\{1, 2, \ldots, |V(H)|-1\}$. Then every other vertex of $H$ belongs to  the set $\{z_1, z_2, \ldots, z_n\}$. However, $\{z_1, z_2, \ldots, z_n\}$ is an independent set and, therefore, any non-empty connected subgraph of $\cP_n(H)$ with vertices in $\{z_1, z_2, \ldots, z_n, k\}$ will be a star-graph, which contradicts the assumption of  the lemma. 
\end{proof}

Now, choose  $c_n=n^{\frac{1}{|V(H)|-1}}$. By the above lemma, 
\begin{align*}
\P(T(H, \cP_n(H)) >0)&=\P(\text{at least two vertices in } \{1, 2, \ldots, |V(H)|-1\} \text{ have the same color}) \nonumber \\
&\leq \frac{{|V(H)|-1 \choose 2}}{c_n} \rightarrow 0,
\end{align*}
as $n \rightarrow \infty$. Therefore, $T(H, \cP_n(H)) \pto 0$. However, the number of monochromatic $H$ in $\ceil{\lambda n}$ disjoint copies of $H$ follows $\dBin(\ceil{\lambda n}, \frac{1}{c_n^{|V(H)|-1}})=\dBin(\ceil{\lambda n}, \frac{1}{n})$, which converges to $\dPois(\lambda)$, as $n \rightarrow \infty$. Therefore, $$T(H, G_n(H)) \dto \dPois(\lambda).$$ 

On the other hand, note that $N(H, G_n(H))=N(H, \cP_n(H))+\ceil{\lambda n}$. Then using $N(H, \cP_n(H)) \geq n$, gives $\E T(H, G_n(H))=\frac{1}{c_n^{|V(H)|-1}} N(H, G_n) \geq \frac{\ceil{\lambda n}+n}{c_n^{|V(H)|-1}} \rightarrow \lambda+1$, that is, \eqref{eq:rcondition} does not hold. 

\section{Applications of Theorem \ref{thm:poisson}}
\label{sec:examples}

In this section we apply Theorem \ref{thm:poisson} in various examples: (1) monochromatic subgraphs in the  Erd\H os-R\'enyi random graph (Section \ref{sec:erpf}), (2) monochromatic cliques in general graphs (Section \ref{sec:mcg}), and (3) connections to the birthday paradox (Section \ref{sec:birthday}).

\subsection{Monochromatic Subgraphs in Erd\H os-R\'enyi Random Graphs} 
\label{sec:erpf}

Theorem \ref{thm:poisson} can be easily extended to random graphs, when the limits in \eqref{eq:expvar} hold in probability, when the graph and its coloring are jointly independent. This is explained in the following lemma, using which we prove Theorem \ref{thm:balanced} and Theorem \ref{thm:unbalanced}, in Section \ref{sec:pferdosrenyi}. 

\begin{lem}\label{lm:random}
Let $\{G_n\}_{n\ge 1}$ be a sequence of random graphs independent of the coloring distribution $(X_1,\cdots, X_{|V(G_n)|})$ such that
$$\E(T(H,G_n)|G_n)\stackrel{P}{\rightarrow}\lambda, \quad \Var(T(H,G_n)|G_n)\stackrel{P}{\rightarrow}\lambda.$$
Then $T(H,G_n)\stackrel{D}{\rightarrow}\dPois(\lambda)$.
\end{lem}

\begin{proof} The given hypothesis implies the existence of positive reals $\varepsilon_n \rightarrow 0$, such that
$$\lim_{n\rightarrow\infty}\P(A_n)=0,\quad A_n:=\left\{G_n: \max\{|\E(T(H,G_n)|G_n)-\lambda|,|\Var(T(H,G_n)|G_n)-\lambda|)>\varepsilon_n\}\right\}.$$
Thus, given any function $h:\Z_{+} \cup \{0 \} \mapsto [0,1]$ 
\begin{align*}
|\E h(T(H,G_n))-\E h(\dPois(\lambda))|\le \P(A_n) +\sup_{G_n\in A_n^c} |\E(h(T(H,G_n))|G_n)-\E h(\dPois(\lambda))|.
\end{align*}
It thus suffices to prove that the second term in the RHS above converges to $0$. If not, there exists a deterministic sequence of graphs $\{G_n'\}_{n\ge 1}$ such that $\E (T(H,G_n'))$ and $\Var(T(H,G_n'))$ both converge to $\lambda$, but $T(H,G_n')$ does not converge to $\dPois(\lambda)$, a contradiction to Theorem \ref{thm:poisson}.
\end{proof}

\subsubsection{\textbf{Proofs of Theorem \ref{thm:balanced} and Theorem \ref{thm:unbalanced}}}\label{sec:pferdosrenyi}

We begin with some preliminary properties of the exponent $\gamma(H)$ (recall \eqref{eq:expp}). 

\begin{lem}\label{lm:egamma}Let $H$ be a connected graph. Then the following hold:
\begin{enumerate} 
\item[(a)] If $H$ is unbalanced, then $\gamma(H)$ is well-defined, and $0< \gamma(H) < \frac{1}{m(H)}$, where $m(H)$ is defined in \eqref{eq:mH}. Furthermore, every minimizer of \eqref{eq:expp} is an induced sub-graph $H_1$ of $H$. 

\item[(b)] If $H$ is balanced, but not strictly balanced, then $\gamma(H) = \frac{1}{m(H)}$.   
\end{enumerate}
\end{lem}

\noindent{\it Proof}. Throughout, we assume $H$ is connected. Then we have the following two cases: 
\begin{enumerate}
\item[(a)] Suppose $H$ is unbalanced. Then there exists $H_1\subset H$ non-empty such that
$$\frac{|E(H_1)|}{|V(H_1)|}>\frac{|E(H)|}{|V(H)|}\Leftrightarrow |E(H_1)||V(H)|-|E(H)||V(H_1)|>0.$$
For this $H_1$, 
\begin{align}\label{eq:gammaH_positive} 
|E(H_1)|& (|V(H)|-1)-|E(H)|(|V(H_1)|-1) \nonumber \\
&=|E(H_1)||V(H)|-|E(H)||V(H_1)|+|E(H)|-|E(H_1)|>0. 
\end{align}
Thus, the minimum in definition of $\gamma(H)$ (recall  \eqref{eq:expp}) is not over an empty set, which means $\gamma(H)$ is well-defined. Moreover, as the minimum is taken over finitely many positive items, $\gamma(H)>0$.

Next, suppose $H_1\subset H$ such that $m(H)=\frac{|E(H_1)|}{|V(H_1)|}$. To show $\gamma(H) < \frac{1}{m(H)}$ it suffices to show that 
\begin{align*}
\frac{|V(H)|-|V(H_1)|}{|E(H_1)|(V(H)-1)-|E(H)|(|V(H_1)|-1)} < \frac{|V(H_1)|}{|E(H_1)|},
\end{align*}
which is equivalent to $|E(H_1)||V(H)|(|V(H_1)|-1)-|V(H_1)| |E(H)| (|V(H_1)|-1) >0$, that is, $|E(H_1)||V(H)| -|E(H)||V(H_1)| >0$, which holds since $H$ is unbalanced. 

Finally, observe that, for fixed $|V(H_1)|$, the RHS in \eqref{eq:expp} is decreasing in $|E(H_1)|$, which implies that every minimizer of \eqref{eq:expp} is an induced subgraph $H_1$ of $H$.
\\

\item[(b)] Now, suppose $H$ balanced, but not strictly balanced. Then there exists a proper subgraph $H_1$ of $H$ such that $m(H)=\frac{|E(H_1)|}{|V(H_1)|}=\frac{|E(H)|}{|V(H)|}$. Then this $H_1$ satisfies \eqref{eq:gammaH_positive}, and, therefore $\gamma(H)$ is well defined, positive, and satisfies
\begin{align}\label{eq:mh_I}
\gamma(H) \leq \frac{|V(H)|-|V(H_1)|}{|E(H_1)|(|V(H)|-1)-|E(H)|(|V(H_1)|-1)} = \frac{|V(H_1)|}{|E(H_1)|}=\frac{1}{m(H)}.
\end{align}
Next, we are going to show that if $H'$ is a subgraph of $H$, such that $
|E(H')|(|V(H)|-1)-|E(H)|(|V(H')|-1)>0$, then 
\begin{align}\label{eq:mh_II}
\frac{|V(H)|-|V(H')|}{|E(H')|(|V(H)|-1)-|E(H)|(|V(H')|-1)} \geq \frac{|V(H)|}{|E(H)|}=\frac{1}{m(H)}.
\end{align}
This is equivalent to showing $|V(H')||E(H)|(|V(H)|-1)-|V(H)||E(H')|(|V(H)|-1) \geq 0$, which follows by noting that $|V(H')||E(H)| \geq |V(H)| |E(H')|$, since $H$ is balanced. Combining \eqref{eq:mh_I} and \eqref{eq:mh_II}, it follows that $\gamma(H)=\frac{1}{m(H)}$, for $H$ which is balanced, but not strictly balanced. \hfill $\Box$ \\ 
\end{enumerate}

\noindent{\textit{Proof of Theorem \ref{thm:unbalanced}(a)}}: Consider the subgraph $H_1$ of $H$ such that the minimum in \eqref{eq:expp} is attained, that is, $$\gamma(H)=\frac{|V(H)|-|V(H_1)|}{|E(H_1)|(|V(H)|-1)-|E(H)|(|V(H_1)|-1)}.$$ Note that  $\P(T(H, G_n)>0) \leq  \P(T (H_1, G_n)>0)  \leq \E(T(H_1, G_n))$. Therefore, 
 \begin{align}
 \P(T(H, G_n)>0) & \leq \E(T(H_1, G_n)) \nonumber \\
 & \lesssim_{H} \frac{n^{|V(H_1)|} p^{|E(H_1)|}}{c_n^{|V(H_1)|-1}}\nonumber\\
 &\lesssim_{H} \frac{n^{|V(H_1)|} p^{|E(H_1)|}}{n^{\frac{|V(H)| (|V(H_1)|-1) }{|V(H)|-1}} p^{\frac{|E(H)| (|V(H_1)|-1)}{|V(H)|-1}}} \tag*{ (using $c_n=\Theta(n^{\frac{|V(H)|}{|V(H)|-1}} p^{\frac{|E(H)|}{|V(H)|-1}})$)}\nonumber \\
& = \left(\frac{n^{|V(H_1)|(|V(H)|-1)} p^{ |E(H_1)|(|V(H)|-1)}}{n^{ |V(H)| (|V(H_1)|-1)} p^{|E(H)| (|V(H_1)|-1)}}\right)^{\frac{1}{|V(H)|-1}} \nonumber \\ 
& = \left( n^{|V(H)|-|V(H_1)|} p^{|E(H_1)|(|V(H)|-1)-|E(H)|(|V(H_1)|-1)}\right)^{\frac{1}{|V(H)|-1}} \label{eq:partb} \\
&=\left( n^{\gamma(H)} p\right)^{\frac{|E(H_1)|(|V(H)|-1)-|E(H)|(|V(H_1)|-1)}{|V(H)|-1}}. \nonumber
\end{align}
Since the RHS above goes to $0$ by assumption, the proof of Theorem \ref{thm:unbalanced}(a) is complete. \\

\noindent{\textit{Proof of Theorem \ref{thm:unbalanced}(b)}}:  
For any integer $r \geq 1$, a direct expansion gives 
$$\E T(H,G_n)^r=\sum_{F\in \sG_r(H)} c_0(F,H) \frac{\E N(F,G_n)}{c_n^{|V(F)|-\nu(F)}}=\sum_{F\in \sG_r(H)}c_1(F,H)\frac{n^{|V(F)|} p^{|E(F)|}}{c_n^{|V(F)|-\nu(F)}},$$
where $c_0(F,H), c_1(F,H)$ are constants free of $n$, and $\sG_r(H)$ is the set of all unlabeled graphs formed by the join of $r$ isomorphic copies of $H$. The convergence of the  moments of $T(H,G_n)$ follows from the lemma below.

\begin{lem}\label{lm:mconv} For $F\in \sG_r(H)$, define $\eta_n(F):=\frac{1}{c_n^{|V(F)|-\nu(F)}} n^{|V(F)|} p^{|E(F)|}$. Then $\eta(F):=\lim_{n \rightarrow \infty}\eta_n(F)$ exists. 
\end{lem}

\begin{proof} First, note that it suffices to prove the lemma for connected $F$, since in the general case, if $F$ has connected components $F_1',\ldots,F_\nu'$, then $\eta_n(F)=\prod_{i=1}^\nu \eta_n(F_i')$. If $F \in \sG_r(H)$, then each $F_i' \in \sG_r(H)$ too, so convergence of each term in the product will show convergence of $\eta_n(F)$.  We proceed by induction on $r$. For $r=1$, $F=H$, and $\eta_n(F) = \frac{1}{c_n^{|V(H)|-1}} n^{|V(H)|} p^{|E(H)|}  \rightarrow \lambda_0 := \lambda |Aut(H)|$, by the assumption \eqref{expcopy}. Now, suppose the result holds for all connected $F\in \sG_{r-1}(H)$, and let $F\in \sG_r(H)$ be connected. Then, $F$ is the join of $F_1$ and $F_2$, for some connected $F_1\in \sG_{r-1}(H)$ and  an isomorphic copy $F_2$ of $H$.

Let $H_1$ be a graph with vertex set $V(F_1)\cap V(F_2)$ and edge set $E(F_1)\cap E(F_2)$. We need to show the convergence of 
$$\eta_n(F_r)=\frac{n^{|V(F_r)|}p^{|E(F_r)|}}{c_n^{|V(F_r)|-\nu(F_r)}}=\frac{n^{|V(F_{1})|} p^{|E(F_{1})|}}{c_n^{|V(F_{1})|-1}}\times \frac{n^{|V(H)|-|V(H_1)|} p^{|E(H)|-|E(H_1)|}}{c_n^{|V(H)|-|V(H_1)|}}.$$ The first term in the RHS above converges by induction hypothesis. For the second term, using  \eqref{expcopy} gives 
\begin{align}\label{eq:partbn}
&\frac{n^{|V(H)|-|V(H_1)|} p^{|E(H)|-|E(H_1)|}}{c_n^{|V(H)|-|V(H_1)|}}  \nonumber \\
& =(1+o(1))\lambda_0 \frac{n^{-|V(H_1)|}p^{-|E(H_1)|}}{c_n^{-(|V(H_1)|-1)}}  \nonumber \\
& =(1+o(1)) \lambda_0^{\frac{|V(H)|-|V(H_1)|}{|V(H)|-1}} \left( n^{|V(H)|-|V(H_1)|} p^{|E(H_1)|(|V(H)|-1)-|E(H)|(|V(H_1)|-1)}\right)^{-\frac{1}{|V(H)|-1}} ,
 \end{align}
which converges to $\lambda_0^{\frac{|V(H)|-|V(H_1)|}{|V(H)|-1}} \kappa^{-\frac{|E(H_1)|(|V(H)|-1)-|E(H)|(|V(H_1)|-1)}{|V(H)|-1}}\in (0,\infty)$, when $H_1$ attains the minimum in \eqref{eq:expp}, since $n^{\gamma(H)} p \rightarrow \kappa \in (0,\infty)$.  Otherwise,
$$n^{\frac{|V(H)|-|V(H_1)|}{|E(H_1)|(|V(H)|-1)-|E(H)|(|V(H_1)|-1)}} p \gg n^{\gamma(H)}p,$$ and so $\eta_n(F)$ converges to $0$.  Thus, $\eta_n(F)$ converges for all connected graphs $F\in \sG_r(H)$, and the proof of Lemma \ref{lm:mconv} is complete.
\end{proof}

Now, if $T(H, G_n) \dto W$ for some random variable $W$, then $\E T(H, G_n)^r \dto \E W^r$, for any integer $r \geq 1$. Thus, to show that $W$ is not a Poisson distribution, it suffices to prove that
\begin{align}\label{eq:nonpoisson}
\liminf_{n\rightarrow\infty}\Var(T(H,G_n))\ge \liminf_{n\rightarrow\infty} \E \Var(T(H,G_n)|G_n)>\lambda.
\end{align}
Recall from \eqref{eq:varT},  $\Var(T(H,G_n)|G_n)=R_{1,n}+R_{2,n},$ where $\E R_{1,n} \rightarrow \lambda$. Therefore, it suffices to show that $\lim\inf_{n \rightarrow \infty} \E R_{2,n} >0$. To show this, let $H_1$ be the subgraph of $H$ for which the minimum in \eqref{eq:expp} is attained, and $F_0$ be the $|V(H_1)|$-join (note that $|V(H_1)|< |V(H)|$ by Lemma \ref{lm:egamma}) of $H$ and $H'$, where $H'$ is isomorphic to $H$, such that $V(H) \cap V(H')=V(H_1)$ and $E(H) \cap E(H')=E(H_1)$. Moreover, let $$\sJ_{\geq 2}(H) := \bigcup_{t=2}^{|V(H)|} \sJ_t(H),$$
where $\sJ_{t}(H)$ is as in Definition \ref{defn:tjoin}. Then, there exist constants $c_2(F, H)$, such that 
\begin{align}
\E R_{2,n}&= (1+o(1))\sum_{F\in \sJ_{\geq 2}(H)} c_2(F,H) \frac{n^{|V(F)|} p^{|E(F)|}}{c_n^{|V(F)|-1}} \nonumber \\
& \geq (1+o(1)) c_2(F_0,H) \frac{n^{|V(F_0)|} p^{|E(F_0)|}}{c_n^{|V(F_0)|-1}} \nonumber \\
& = (1+o(1)) c_2(F_0,H) \lambda_0^{\frac{2|V(H)|-|V(H_1)|-1}{|V(H)|-1}} \left(n^{\gamma(H)} p\right)^{-\frac{|E(H_1)|(|V(H)|-1)-|E(H)|(|V(H_1)|-1)}{|V(H)|-1}} \nonumber \tag*{(using \eqref{eq:partbn})} \\
& = (1+o(1)) c_2(F_0,H) \lambda_0^{\frac{2|V(H)|-|V(H_1)|-1}{|V(H)|-1}} \kappa^{-\frac{|E(H_1)|(|V(H)|-1)-|E(H)|(|V(H_1)|-1)}{|V(H)|-1}} >0. \nonumber  
\end{align}
This implies \eqref{eq:nonpoisson}, completing the proof of Theorem \ref{thm:unbalanced}(b). \\

\noindent{\textit{Proofs of Theorem \ref{thm:balanced}(a) and Theorem \ref{thm:unbalanced}(c)}}:  Note that, in this regime, $p \gg n^{-\frac{1}{m(H)}}$ (by Lemma \ref{lm:egamma}), which implies $N(H, G_n)=(1+o_P(1)) \E(N(H, G_n))$. Therefore, $$\E (T(H, G_n)|G_n)=\frac{1}{c_n^{|V(H)|-1}} N(H, G_n)=(1+o_P(1))\frac{1}{c_n^{|V(H)|-1}} \E N(H, G_n) =(1+o_P(1))\lambda,$$ by assumption \eqref{expcopy}. Therefore, by Lemma \ref{lm:random}, it suffices to check that $\Var(T(H, G_n)|G_n) \pto \lambda$, which is equivalent to $N(F, G_n)=o_P(c_n^{2|V(H)|-t-1})$, for every $F \in \sJ_t(H)\setminus \{H\}$ and $t \in [2, |V(H)|]$. Since, $|V(F)|={2|V(H)|-t}$, it suffices to show that 
\begin{align}\label{eq:FerI}
\E N(F,G_n)=o(c_n^{|V(F)|-1}), \quad \text{ for all connected } F \ne H,
\end{align}
formed by the join of $H$ and another isomorphic copy $H'$. To this end, define $H_1=(V(H)\cap V(H'), E(H) \cap E(H'))$, which is a (possibly disconnected) subgraph of $H$. Then $|V(F)|=2|V(H)|-|V(H_1)|$, $|E(F)|=2|E(H)|-|E(H_1)|$, and 
\begin{align}
\frac{\E(N(F, G_n))}{c_n^{|V(F)|-1}}  & \lesssim_F \frac{n^{|V(F)|}p^{|E(F)|}}{c_n^{|V(F)|-1}} \nonumber \\
&= \frac{n^{2|V(H)|-|V(H_1)|} p^{2|E(H)|-|E(H_1)|}}{c_n^{2|V(H)|-|V(H_1)|-1}} \nonumber \\
&\lesssim \lambda_0^2(1+o(1))\frac{n^{-|V(H_1)|} p^{-|E(H_1)|}}{c_n^{-|V(H_1)|+1}} \tag*{(using \eqref{expcopy})}  \nonumber \\ 
\label{eq:FerII}&\lesssim\frac{c_n^{|V(H_1)|-1}}{n^{|V(H_1)|}p^{|E(H_1)|}}.
\end{align}
Therefore, to establish \eqref{eq:FerI}, it suffices  to verify that the RHS above goes to zero, as $n \rightarrow \infty$, for every connected $F \ne H$ formed by the join of two isomorphic copies of $H$. 

Now, using $c_n=\Theta(n^{\frac{|V(H)|}{|V(H)|-1}} p^{\frac{|E(H)|}{|V(H)|-1}})$, as in \eqref{eq:partb}, the RHS of \eqref{eq:FerII} becomes 
$$ \left( n^{|V(H)|-|V(H_1)|} p^{|E(H_1)|(|V(H)|-1)-|E(H)|(|V(H_1)|-1)}\right)^{-\frac{1}{|V(H)|-1}}.$$ Therefore, it suffices to show that 
\begin{align}\label{eq:partb2}
n^{|V(H)|-|V(H_1)|} p^{|E(H_1)| (|V(H)|-1)-|E(H)|(|V(H_1)|-1)}\rightarrow\infty.
\end{align}

Now, depending on whether $H$ is balanced or not, we consider two cases: 

\begin{itemize}

\item {\it $H$ is balanced}: In this case, $n^{-\frac{|V(H)|}{|E(H)|}} \ll p \ll 1$. Using $\frac{|E(H_1)|}{|V(H_1)|}\le \frac{|E(H)}{|V(H)|}$, the LHS of \eqref{eq:partb2} becomes 
$$n^{|V(H)|-|V(H_1)|} p^{|E(H_1)| (|V(H)|-1)-|E(H)|(|V(H_1)|-1)} \geq \left(n p^{\frac{|E(H)|}{|V(H)|}}\right)^{|V(H)|-|V(H_1)|},$$
which implies \eqref{eq:partb2}, whenever $|V(H_1)|<|V(H)|$, since $n p^{\frac{|E(H)|}{|V(H)|}} \rightarrow \infty$ by assumption. 

Otherwise, assume $|V(H_1)|=|V(H)|$, in which case the LHS of \eqref{eq:partb2} becomes\\ $ p^{(|E(H_1)| -|E(H)|)(|V(H)|-1)} \rightarrow \infty$, whenever $|E(H_1)|<|E(H)|$, since $p \rightarrow 0$. Finally, note that $|V(H_1)|=|V(H)|$ and $|E(H_1)|=|E(H)|$, implies $F=H$ which is impossible, by assumption. Therefore, \eqref{eq:FerI} holds, and by Lemma \ref{lm:random}, $T(H, G_n) \dto \dPois(\lambda)$, completing the proof of Theorem \ref{thm:balanced}(a). 

\item {\it $H$ is unbalanced}: In this case, $n^{-\gamma(H)} \ll p \ll 1$. 
\begin{itemize}
\item If $|E(H_1)| (|V(H)|-1)-|E(H)|(|V(H_1)|-1) < 0$, then \eqref{eq:partb2} is obvious, since $p:=p(n) \rightarrow 0$ and $|V(H_1)|\leq |V(H)|$.

\item If $|E(H_1)| (|V(H)|-1)-|E(H)|(|V(H_1)|-1) > 0$ (this implies $|V(H)|>|V(H_1)|$), then by the definition of $\gamma(H)$ (see \eqref{eq:expp}), $$\gamma(H)\le \frac{|V(H)|-|V(H_1)|}{|E(H_1)| (|V(H)|-1)-|E(H)|(|V(H_1)|-1)},$$ and so
\begin{align*}
n^{|V(H)|-|V(H_1)|} & p^{|E(H_1)| (|V(H)|-1)-|E(H)|(|V(H_1)|-1)} \nonumber \\ 
& \ge n^{|V(H)|-|V(H_1)|} p^{\frac{|V(H)|-|V(H_1)|}{\gamma(H)}} \nonumber \\ 
& =(n^{\gamma(H)}p)^{\frac{|V(H)|-|V(H_1)|}{\gamma(H)}},
\end{align*}
which implies \eqref{eq:partb2}, since $n^{\gamma(H)}p\rightarrow\infty$ by assumption.

\item If $|E(H_1)| (|V(H)|-1)-|E(H)|(|V(H_1)|-1)=0$, but $|V(H_1)|< |V(H)|$, then again  \eqref{eq:partb2} is obvious. Otherwise, $$|V(H_1)|=|V(H)| \text{ and } |E(H_1)| (|V(H)|-1)-|E(H)|(|V(H_1)|-1) =0.$$  
This implies $|E(H_1)|=|E(H)|$, and hence, $H = F$, which is impossible, by assumption. 
\end{itemize}
\end{itemize}
This implies \eqref{eq:FerI}, and hence by Lemma \ref{lm:random}, $T(H, G_n) \dto \dPois(\lambda)$, completing the proof of Theorem \ref{thm:unbalanced}(c).\\

\noindent{\textit{Proofs of Theorem \ref{thm:balanced}(b) and Theorem \ref{thm:unbalanced}(d)}}:  
Finally, if $p(n) := p \in (0, 1)$ is fixed,  $G_n$ converges to the constant graphon $W^{(p)}=p$, and $$\frac{N_{\mathrm{ind}}(F, G_n)}{c_n^{|V(F)|-1}}=\frac{\lambda_0(1+o(1))}{|Aut(F)|} p^{|E(F)|-|E(H)|}(1-p)^{{|V(H)| \choose 2}-|E(F)|},$$ 
for every super-graph $F$ of $H$ with $|V(F)|=|V(H)|$. The results in Theorem \ref{thm:balanced}(b) and Theorem \ref{thm:unbalanced}(d), then follows from Theorem \ref{thm:poisson_linear} and Remark \ref{rm:dense}.

\subsection{Monochromatic Cliques} 
\label{sec:mcg}

Assumption \eqref{eq:expvar} of Theorem \ref{thm:poisson} is equivalent to the conditions:
\begin{itemize}
\item[(1)] $\frac{1}{c_n^{|V(H)|-1}}N(H,G_n)\rightarrow \lambda$, and

\item[(2)]  $N(F, G_n)=o(c_n^{2|V(H)|-t-1})$, for every $F \in \sJ_t(H)\setminus \{H\}$ and $t \in [2, |V(H)|]$. 
\end{itemize}

These conditions simplify considerable when $H=K_s$ is the $s$-clique. To this end, note that, for every $t \in [2, s-1]$, because of the symmetry of $K_s$, all $t$-joins of $K_s$ are isomorphic, that is, $\sJ_t(K_s)=\{J_t(K_s)\}$, where $J_t(K_s)$ is the graph obtained by the superimposition of two isomorphic copies of $K_s$, such that the two vertex sets intersect at exactly $t$ vertices. Therefore, for $t \in [2,s-1]$, condition (2) above simplifies to, 
\begin{align}\label{eq:clique}
N(J_t(K_s), G_n)=o(c_n^{2s-t-1})=o(N(K_s, G_n)^\frac{2s-t-1}{s-1}), \quad  \text{for every} \quad  t \in [2, s-1],
\end{align}
using $\E(T(K_s, G_n))=\frac{1}{c_n^{s-1}} N(K_s, G_n) \rightarrow \lambda$.
Moreover, the set $\sJ_s(K_s) \backslash \{K_s\}$ is empty, and condition (2) above, for the case $t = s$, is trivially true. Therefore, we have the following corollary:

\begin{cor}\label{cor:clique}
$T(K_s, G_n)\dto \dPois(\lambda)$ whenever $\E(T(K_s, G_n))\rightarrow \lambda$ and \eqref{eq:clique} holds. \qed
\end{cor}

In particular, when $H=K_3$ is the triangle, the above corollary implies, $T(K_3, G_n)\dto \dPois(\lambda)$ whenever $\E(T(K_3, G_n))\rightarrow \lambda$ and $N(\cD, G_n)=o(N(K_3, G_n)^\frac{3}{2})$, where $\cD$ is the {\it diamond}: the 4-cycle with a diagonal. 

\begin{remark}\label{rm:trianglestein} As mentioned before, Theorem \ref{thm:poisson}, and, in particular, Corollary \ref{cor:clique}, does not follow by applying the Stein's method using a generic dependency graph \cite{poisson2,CDM}. For example, let $H=K_3$ be the triangle and denote by $\sX_3$ the set of 3-element subsets of $V(G_n)$ which form a triangle in $G_n$. Then the graph with vertex set $\sX_3$ which puts an edge  between two elements in $\sX_3$ whenever they are non-overlapping, is a valid dependency graph for the collection $(\bm 1\{X_{=\bm s}\})_{\bm s \in \sX_3}$. 
Now, if $\E(T(K_3, G_n))\rightarrow \lambda$, using this dependency graph in \cite[Theorem 15]{CDM}, shows that $T(K_3, G_n) \dto \dPois(\lambda)$, if $$N(\cD, G_n)=o(N(K_3, G_n)^\frac{3}{2}) \quad \text{and} \quad N(\bowtie, G_n)=o(N(K_3, G_n)^\frac{3}{2}),$$ where $\bowtie$ denotes two triangles joined at a vertex. This condition is, in general, stronger than Corollary \ref{cor:clique}: For instance, in the wheel graph on $W_n$ on $n$-vertices,\footnote{The wheel graph $W_n$ has vertex-set $V(W_n):=\{0, 1,2,\ldots, n\}$, and edge-set $E(W_n)=\{ (0,1), (0,2), \ldots, (0,n),$ $(1,2), (2,3), \ldots, (n-1,n), (n,1)\}$.} colored with $c_n$ colors such that $\E(T(K_3, W_n))=\frac{n}{c_n^3}\rightarrow 1$,  it is easy to check that $T(K_3, W_n) \dto \dPois(1)$, but $N(\bowtie, W_n)=\frac{n(n-1)}{2}$, that is, the above dependency graph construction does not work. This is because, unlike the direct moment-based approach, the generic dependency graph construction is unable to leverage the fact that the $\Cov(\bm 1\{X_{=\bm s}\}, \bm 1\{X_{=\bm t}\})=0$, whenever $\bm s, \bm t \in \sX_3$ have 1 vertex index in common. It would be interesting to see whether a more sophisticated dependency graph construction or other versions of Stein's method can be used to prove Theorem \ref{thm:poisson}, and obtain rates of convergence. 
\end{remark}

\subsection{Birthday Problem} 
\label{sec:birthday}

The case $H=K_s$ is the $s$-clique, is of particular interest, because it generalizes the well-known birthday problem to a general friendship network $G_n$. In the birthday problem, $G_n$ is a friendship-network graph where the vertices are colored uniformly with $c_n=365$ colors (corresponding to birthdays). In this case, two friends will have the same birthday whenever the corresponding edge in the graph $G_n$ is monochromatic.  Therefore, $\P(T(K_{s}, G_n)>0)$ is the probability that there is an $s$-{\it fold birthday match}, that is, there are $s$ friends with the same birthday. For example, if the network $G_n$ satisfies \eqref{eq:clique}, Corollary \eqref{cor:clique} implies 
\begin{align}\label{eq:cliquebirthday}
\P(T(K_{s}, G_n)>0) \approx 1-\exp\left(-\frac{N(K_s, G_n)}{c_n^{s-1}}  \right) = p,
\end{align}
from which we can compute the approximate number of people needed to ensure a  $s$-fold birthday match in the network $G_n$, with probability at least $p$. 

\begin{itemize}

\item[--] In the classical birthday problem, the underlying graph $G_n=K_n$ is the complete graph. In this case, $N(K_s, G_n)={n \choose s}$. For example, using $p=\frac{1}{2}$, $s=4$, and $c_n=365$ in \eqref{eq:cliquebirthday}, gives that in any group of approximately 167 people, with probability at least 50\%, there are four friends all having the same birthday. Diaconis and Mosteller \cite{diaconismosteller} considered the following related example: Suppose a friend reports that she, her husband, and their daughter were all born on the same day of the month (say the 16th). Taking $c_n = 30$ (days in a month), $s = 3$, and $p = \frac{1}{2}$, in \eqref{eq:cliquebirthday} gives that among birthdays of 16 people, a triple match in day of the month has about 50\% chance.

\item[--] Another interesting case is birthday coincidences among different types, for example, with two types (boy/girl) one can ask what is the chance there is a boy-girl birthday match among a group of $n$ boys and $n$ girls? More generally, with $s$-types and $n$ objects in each type, an $s$-fold birthday coincidence corresponds to an $s$-clique in the complete $s$-partite graph with $n$ vertices in each part. For example, using  $N(K_3, K_{n, n, n}) =n^3$ and substituting $p=0.5$, $s=3$, $c_n=365$ in the formula gives, in any collection of 3 types (say nationality, for example, American, French, and Indian) of approximately 45 people each, with probability at least 50\%, there is a triple birthday match, that is, an American, a French, and an Indian, have the same birthday.   Asymptotics of collision times among different objects are useful in developing algorithms for the discrete logarithm problem \cite{ghdl}.

\end{itemize}

\subsection*{Acknowledgements} The authors are grateful to the anonymous referee for the detailed and insightful comments, which greatly improved the quality and the presentation of the paper.

\appendix

\section{The Ordering Lemma}

In this section we prove the ordering lemma used in the proof of Lemma \ref{lm:epscount}. To this end, let $\Omega \subset \N$ be finite and $R \geq 3$ a non-negative integer. A  collection $\bm S=(\bm s_1, \bm s_2, \ldots, \bm s_N)$, where $\bm s_j \in \Omega_R$, for $j \in [N]$, is said to be {\it connected} if there exists an {\it ordering} (permutation) $\sigma: [N] \rightarrow [N] $ such that 
$$\sX_{\bm S}(t, \sigma):=\left|\bar{\bm s}_{\sigma(t)} \bigcap \left(\bigcup_{a=1}^{t-1} \bar{\bm s}_{\sigma(a)} \right) \right|  \geq 1,$$
for every $t \in [2, N]$. 

\begin{lem}\label{connected} Suppose $\bm S = (\bm s_1, \bm s_2, \ldots, \bm s_N) \in \Omega_R^N$ is connected, and $|\bigcup_{j=1}^N \bar{\bm s}_j| <  b R-b+1$, where $b$ is the number of distinct $R$-element sets in the collection $\{\bar{\bm s}_1, \bar{\bm s}_2, \ldots, \bar{\bm s}_N\}$. Then there exists an ordering $\sigma: [R] \rightarrow [R] $ such that the following hold: 
\begin{itemize}
\item[--] $\sX_{\bm S}(t, \sigma)  \geq 1$, for every $t \in [2, N]$, and  
\item[--] $\sX_{\bm S}(t, \sigma) \in [2, R-1]$, for some $t \in [2, N]$. 
\end{itemize}
\end{lem}

\begin{proof}
Since $\bm S$ is connected, there exists an ordering $\sigma$, such that $\sX_{\bm S}(t,\sigma) \geq 1$ for every $t \in [2,N]$. Suppose that for every $t \in [2,N]$, $\sX_{\bm S}(t,\sigma) \in \{1,R\}$, and, towards a contradiction, assume that for every $2\leq t\leq N$, either $\sX_{\bm S}(t,\sigma) = 1$, or $\bar{\bm s}_{\sigma(t)} \in \{\bar{\bm s}_{\sigma(1)},\dots, \bar{\bm s}_{\sigma(t-1)}\}$. Define
$$k=\big|\{t\in[2,N]: \sX_{\bm S}(t,\sigma) = 1\}\big| .$$ Then, $b = 1+k$ and $|\bigcup_{j=1}^N \bar{\bm s}_j| = R + k(R-1)$. This yields a contradiction, because $$\Big|\bigcup_{j=1}^N \bar{\bm s}_j\Big| = R + (b-1)(R-1) = bR-b+1 .$$ 	

Hence, there exists $2\leq t\leq N$ such that $\sX_{\bm S}(t,\sigma)=R$ and $\bar{\bm s}_{\sigma(t)} \notin \{\bar{\bm s}_{\sigma(1)},\dots, \bar{\bm s}_{\sigma(t-1)}\}$. Define 
\begin{equation*}%\label{t0}
t_0 = \inf \big\{2\leq t\leq N:\sX_{\bm S}(t,\sigma)=R~\textrm{and}~\bar{\bm s}_{\sigma(t)} \notin \{\bar{\bm s}_{\sigma(1)},\dots, \bar{\bm s}_{\sigma(t-1)}\} \big\} ~\textrm{and}
\end{equation*}
\begin{equation*}%\label{t1}
t_1 = \inf \big\{1\leq t < t_0: \bar{\bm s}_{\sigma(t_0)}\cap \bar{\bm s}_{\sigma(t_1)} \neq \emptyset \} .
\end{equation*} 
Clearly, there exists a permutation $\tau: [R]\rightarrow [R]$ such that $\tau(1) = \sigma(t_0)$, $\tau(2) = \sigma(t_1)$ and $\sX_{\bm S}(t,\tau) \geq 1$ for every $t \in [2,N]$. By the definition of $t_0$, it follows that $\bar{\bm s}_{\sigma(t_1)} \neq \bar{\bm s}_{\sigma(t_0)}$, and hence, if $\big|\bar{\bm s}_{\sigma(t_0)}\cap \bar{\bm s}_{\sigma(t_1)}\big| \geq 2$, then $\sX_{\bm S}(2,\tau) \in [2,R-1]$, as required.

So, suppose that $\big|\bar{\bm s}_{\sigma(t_0)}\cap \bar{\bm s}_{\sigma(t_1)}\big| = 1$, and let $\{s\} = \bar{\bm s}_{\sigma(t_0)}\cap \bar{\bm s}_{\sigma(t_1)}$. Define 
\begin{equation*}%\label{t2}
t_2 = \inf \big\{t_1< t < t_0: (\bar{\bm s}_{\sigma(t_0)}\setminus \{s\})\cap \bar{\bm s}_{\sigma(t)} \neq \emptyset \} .
\end{equation*}	
Once again, there exists a permutation $\kappa: [R]\rightarrow [R]$ such that $\kappa(1) = \sigma(t_0)$, $\kappa(2) = \sigma(t_2)$ and $\sX_{\bm S}(t,\kappa) \geq 1$ for every $t \in [2,N]$. So, if $\big|\bar{\bm s}_{\sigma(t_0)}\cap \bar{\bm s}_{\sigma(t_2)}\big| \geq 2$, then $\sX_{\bm S}(2,\kappa) \in [2,R-1]$, as desired.

Hence, assume that $\big|\bar{\bm s}_{\sigma(t_0)}\cap \bar{\bm s}_{\sigma(t_2)}\big| = 1$. Now, there exists a permutation $\theta: [R]\rightarrow [R]$ satisfying: 
 
	\[   
	\theta(t) = 
	\begin{cases}
	\sigma(t) &\quad\text{if} ~1\leq t\leq t_2\\
	\sigma(t_0) &\quad\text{if} ~t=t_2+1\\ 
	\end{cases}
	\]
and $\sX_{\bm S}(t,\theta) \geq 1$ for every $t \in [2,N]$.	Now, it is easy to see that $\sX_{\bm S}(t_2+1,\theta) = 2$, completing the proof of lemma \ref{connected}.
\end{proof}

\section{Convergence Under Double Limit}

Here, we prove the lemma which establishes distributional convergence from moment convergence under the double limit as in \eqref{eq:T_moments}. 

\begin{lem}\label{lm:characteristic_function}
Suppose $\{X_{n, \varepsilon}\}_{n \geq 1, \varepsilon >0}$ be a sequence of real-valued random variables satisfying, for every integer $r\geq 1$, 
\begin{align}\label{eq:XZ}
\limsup_{\varepsilon\rightarrow 0}\limsup_{n\rightarrow\infty}|\E(X_{n, \varepsilon}^r)-\E(Z^r)|=0,
\end{align} 
where $Z$ is a random variable with $\E e^{t Z}<\infty$, for any $t\in \R$. Then for any $t\in \R$, we have $$\limsup_{\varepsilon\rightarrow 0}\limsup_{n\rightarrow\infty}|\E(e^{itX_{n, \varepsilon}})-\E(e^{itZ})|=0,$$ 
that is, $X_{n, \varepsilon}$ converges in distribution to $Z$, as $n \rightarrow \infty$ followed by $\varepsilon \rightarrow 0$. 
\end{lem}

\begin{proof} Fix $K \geq 2$ even. Then for $t\in \R$, by a Taylor's series expansion, 
$$\left|e^{it}-\sum_{s=0}^{K-1}\frac{(it)^s}{s!} \right|\le \frac{|t|^K}{K!}.$$
Using this along with triangle inequality gives,
$$|\E(e^{itX_{n, \varepsilon}})-\E(e^{itZ})|\le \sum_{s=0}^{K-1} \frac{|t|^s}{s!}|\E(X_{n, \varepsilon}^s)-\E(Z^s)|+\frac{|t|^{K}}{K!}\Big\{\E|X_{n, \varepsilon}|^{K}+\E|Z|^{K}\Big\}.$$
On letting $n\rightarrow\infty$ followed by $\varepsilon\rightarrow0$ and using \eqref{eq:XZ} gives, 
$$\limsup_{\varepsilon\rightarrow 0}\limsup_{n\rightarrow\infty}|\E(e^{itX_{n, \varepsilon}})-\E(e^{itZ})|\le \frac{2|t|^{K}}{K!}\E|Z|^{K},$$
From this, the desired conclusion follows by taking limit $K \rightarrow \infty$, along the  even integers, on both sides, since $\E(e^{|tZ|})\le \E(e^{tZ})+\E(e^{-tZ})<\infty$, and recalling that $\E e^{|t Z|}=\sum_{s=0}^\infty \frac{|t|^s}{s!}\E|Z|^s.$
\end{proof}

\end{document}